\documentclass[12pt]{amsart}
\usepackage[latin1]{inputenc}
\usepackage{amsbsy}
\usepackage{amscd} 
\usepackage{amsfonts}
\usepackage{amsgen} 
\usepackage{amsmath}
\usepackage{amsopn} 
\usepackage{amssymb}
\usepackage{amstext}
\usepackage{amsthm} 
\usepackage{amsxtra}
\usepackage[all]{xy}
\usepackage{booktabs}
\usepackage{rotating}

\usepackage{scalerel}

\newcommand\reallywidehat[1]{\arraycolsep=0pt\relax%
\begin{array}{c}
\stretchto{
  \scaleto{
    \scalerel*[\widthof{\ensuremath{#1}}]{\kern-.5pt\bigwedge\kern-.5pt}
    {\rule[-\textheight/2]{1ex}{\textheight}} 
  }{\textheight} %
}{0.5ex}\\           
#1\\                 
\rule{-1ex}{0ex}
\end{array}
}

\theoremstyle{plain} 
\newtheorem{thm}{Theorem}[section]

\newtheorem{prop}[thm]{Proposition}
\newtheorem{lem}[thm]{Lemma}

\theoremstyle{definition}
\newtheorem{defn}[thm]{Definition}
\newtheorem{rem}[thm]{Remark}
\newtheorem{ex}[thm]{Example}

\numberwithin{equation}{section}

\renewcommand{\theta}{\vartheta}
\renewcommand{\phi}{\varphi}
\renewcommand{\epsilon}{\varepsilon}
\renewcommand{\subset}{\subseteq}

\newcommand{\N}{\mathbb N}
\newcommand{\Z}{\mathbb Z}

\DeclareMathOperator{\Aut}{Aut}

\newcommand{\QBan}{G_{aut}^+}
\newcommand{\QBic}{G_{aut}^*}

\begin{document}
\title{On the quantum symmetry of distance-transitive graphs}
\author{Simon Schmidt}
\address{Saarland University, Fachbereich Mathematik, 
66041 Saarbr\"ucken, Germany}
\email{simon.schmidt@math.uni-sb.de}
\thanks{The author is supported by the DFG project \emph{Quantenautomorphismen von Graphen}. He thanks his supervisor Moritz Weber for proofreading the article and many helpful comments and suggestions. This article is part of the author's PhD thesis.}
\date{\today}

\begin{abstract}
In this article, we study quantum automorphism groups of distance-transitive graphs. We show that the odd graphs, the Hamming graphs $H(n,3)$, the Johnson graphs $J(n,2)$ and the Kneser graphs $K(n,2)$ do not have quantum symmetry. We also give a table with the quantum automorphism groups of all cubic distance-transitive graphs. Furthermore, with one graph missing, we can now decide whether or not a distance-regular graph of order $\leq 20$ has quantum symmetry. Moreover, we prove that the Hoffman-Singleton graph has no quantum symmetry. On a final note, we present an example of a pair of graphs with the same intersection array (the Shrikhande graph and the $4 \times 4$ rook's graph), where one of them has quantum symmetry and the other one does not. 
\end{abstract}

\maketitle

\section*{Introduction}

The symmetries of a finite graph are captured by its automorphisms. A graph automorphism is a bijection $\sigma:V \to V$ on the vertices, where vertices $i$ and $j$ are adjacent if and only if $\sigma(i)$ and $\sigma(j)$ are. Via composition, we get a group structure on the set of automorphisms and obtain the automorphism group of the graph. Banica and Bichon generalized the concept of automorphism groups of finite graphs to the so called quantum automorphism groups of finite graphs (\cite{QBan, QBic}) in the framework of Woronowicz's compact matrix quantum groups (\cite{CMQG2}). Here we say that a graph has no quantum symmetry if the quantum automorphism group coincides with its classical automorphism group. In \cite{BanBic}, Banica and Bichon computed the quantum automorphism groups of vertex transitive graphs up to order eleven, except the Petersen graph. There are also results for the quantum automorphism groups of circulant graphs, see \cite{ArthurQSymm, Che}. 

In recent work \cite{QAutPetersen} the author showed that the Petersen graph has no quantum symmetry. Here it was crucial to use the fact that the Petersen graph is strongly regular. Strongly regular graphs are special cases of distance-regular graphs, namely those with diameter two. 
A distance-transitive graph is a graph such that for any given pair of vertices $i,j$ in distance $a$ and any other pair of vertices $k,l$ with $d(k,l)=a$ there is a graph automorphism $\sigma:V \to V$ with $\sigma(i)=k$ and $\sigma(j)=l$. Since distance-transitive graphs are also distance-regular, we can try to apply similar techniques as for the Petersen graph. In the present work, we are not only using the fact that distance-transitive graphs are distance-regular -- the distance-transitivity plays an important role, too. 

Quantum automorphism groups are in close relation to quantum isomorphisms appearing in quantum information theory. This was discovered for example in \cite{nonlocal} and \cite{MustoReutterVerdon}. By results in \cite[Corollary 3.7, Corollary 4.15]{MustoReutterVerdon}, one can classify the quantum graphs and the classical graphs that are quantum isomorphic to a given graph, if this graph has no quantum symmetry. Since we obtain many new examples of graphs without quantum symmetry in this article, this could be used to better understand the quantum graphs and the classical graphs quantum isomorphic to those graphs. 

\section{Main results}
In this section, we summarize the results we obtain in this article. The tools for proving that a graph has no quantum symmetry can be found in Section \ref{tools} whereas the explicit examples are tackled in the Sections \ref{families}, \ref{cubicdist} and \ref{further}. 

Amongst the tools we want to highlight the following two.

\begin{thm} Let $\Gamma$ be a finite, undirected graph, let $(u_{ij})_{1 \leq i,j \leq n}$ be the generators of $C(\QBan(\Gamma))$ and let $d(s,t)$ be the distance of $s, t \in V$.
\begin{itemize}
\item[(i)] If we have $d(i,k) \neq d(j,l)$, then $u_{ij}u_{kl} = 0$.
\item[(ii)] Let $\Gamma$ be distance-transitive. Let $j_1, l_1 \in V$ and put $m:=d(j_1,l_1)$. If $u_{aj_1}u_{bl_1} =u_{bl_1}u_{aj_1}$ for all $a,b$ with $d(a,b)=m$, then we have $u_{ij}u_{kl} =u_{kl}u_{ij}$ for all $i,k, j,l$ with $d(j,l)=m =d(i,k)$.
\end{itemize}
\end{thm}

We show that the odd graphs $O_k$ have no quantum symmetry in Subsection \ref{odd}. Thus, we have an example of an infinite family of distance-transitive graphs having no quantum symmetry, besides the cycles. 

\begin{thm}\label{thm1.1}
The odd graphs $O_k$ have no quantum symmetry.
\end{thm}

The next main result addresses Hamming graphs, see Subsection \ref{Hamming}. We know precisely for which values the Hamming graphs have no quantum symmetry. 

\begin{thm}\label{thm1.2}
The Hamming graphs $H(n,3)$, $n \in \N$, $H(1,2)$ and $H(m,1)$, $m=1,2,3$, have no quantum symmetry. For all other values, the graph $H(n,k)$ has quantum symmetry.
\end{thm}

We also study the Johnson graphs $J(n,2)$ and their complements. In contrary to the Hamming graphs, it is not clear what happens for the Johnson graphs $J(n,k)$ with $k >2$. 

\begin{thm}\label{thm1.3}
For $n\geq 5$, the Johnson graphs $J(n,2)$ and the Kneser graphs $K(n,2)$ do not have quantum symmetry. 
\end{thm}

We also present a table with the quantum automorphism groups of all cubic distance-transitive graphs. There are twelve such graphs as shown by Biggs and Smith in \cite{BiggsSmith}. For three of them, namely the complete graph on four points $K_4$, the complete bipartite graph on six points $K_{3,3}$ as well as the cube $Q_3$, it is known that they have quantum symmetry and their quantum automorphism groups are given in \cite{BanBic}. We show that the remaining ones have no quantum symmetry in Section \ref{cubicdist}, the example of the Petersen graph being known from \cite{QAutPetersen}. We also give the intersection arrays (see Definition \ref{interarray}) of the graphs in the table. This can be used for example to look up if Lemma \ref{ia} applies. 

\begin{thm}\label{mainthm}
Let $\Gamma$ be a cubic distance-transitive graph of order $\geq 10$. Then $\Gamma$ has no quantum symmetry.
\end{thm}

We obtain the following table for the cubic distance-transitive graphs. 

\small
\begin{table}[h]
\begin{center}
\begin{tabular}{lclll}
\toprule
Name of $\Gamma$&Order&$\Aut(\Gamma)$&$\QBan(\Gamma)$&Intersection array\\[.5ex]
\midrule
$K_4$ (\cite{BanBic})&4&$S_4$&$S_4^+$&\{3;1\}\\
$K_{3,3}$ (\cite{BanBic})&6&$ S_3 \wr \Z_2$& $  S_3\wr_*\Z_2$&\{3,2;1,3\}\\
Cube $Q_3$ (\cite{BanBic})&8&$S_4 \times \Z_2$&$S_4^+ \times \Z_2$&\{3,2,1;1,2,3\}\\
Petersen graph(\cite{QAutPetersen})&10&$S_5$&$\Aut(\Gamma)$&\{3,2;1,1\}\\
Heawood graph&14&$PGL(2,7)$&$\Aut(\Gamma)$&\{3,2,2;1,1,3\}\\
Pappus graph&18&ord 216&$\Aut(\Gamma)$&\{3,2,2,1;1,1,2,3\}\\
Desargues graph&20&$S_5\times \Z_2$&$\Aut(\Gamma)$&\{3,2,2,1,1;1,1,2,2,3\}\\
Dodecahedron&20&$A_5\times \Z_2$&$\Aut(\Gamma)$&\{3,2,1,1,1;1,1,1,2,3\}\\
Coxeter graph&28&$PGL(2,7)$&$\Aut(\Gamma)$&\{3,2,2,1;1,1,1,2\}\\
Tutte $8$-cage&30&$\Aut(S_6)$&$\Aut(\Gamma)$&\{3,2,2,2;1,1,1,3\}\\
Foster graph&90&ord $4320$&$\Aut(\Gamma)$&\{3,2,2,2,2,1,1,1;1,1,1,1,2,2,2,3\}\\
Biggs-Smith graph&102&$PSL(2,17)$&$\Aut(\Gamma)$&\{3,2,2,2,1,1,1;1,1,1,1,1,1,3\}\\
\bottomrule
\end{tabular}
\end{center}
\vspace{0.1cm}
\caption{Quantum automorphism groups of all cubic distance-transitive graphs.}\label{table1}
\end{table}
\normalsize


In addition to the previous table, we study more distance-transitive graphs, preferably of order $\leq 20$. As mentioned in the introduction, distance-transitive graphs are distance-regular. There is only one distance-regular graph of order $\leq 20$ that is not distance-transitive, namely the Shrikhande graph. We show that this graph has no quantum symmetry in Subsection \ref{nosym}. 

Except for the Johnson graph $J(6,3)$, we now know for all distance-regular graphs with up to $20$ vertices whether or not they have quantum symmetry. Distance regular graphs of order $11 \leq n \leq 20$ that have quantum symmetry are the $4 \times 4$ rook's graph (Proposition \ref{Hamm}), the $4$-cube (\cite{Hyperoctahedral}), the Clebsch graph (\cite{foldedcube}) as well as the complete graphs $K_n$, the cycles $C_n$, the complete bipartite graphs $K_{n,n}$ and the crown graphs $(K_n \square K_2)^{c}$ for suitable $n$ (see \cite{BanBic} for those four families). The graphs of order $n \leq 11$ can be found in \cite{BanBic}.

It was already shown in \cite{BanBic} that the Paley graph $P_9$ has no quantum symmetry and in \cite{ArthurQSymm} that $P_{13}$ and $P_{17}$ have no quantum symmetry. We give alternative proofs of these facts in Subsection \ref{paley}. 

Besides graphs of order $\leq 20$, we also show that the Hoffman-Singleton graph has no quantum symmetry by proving that all strongly regular graphs with girth five have no quantum symmetry in Subsection \ref{moore}.

We get the following table, containing all distance-regular graphs with up to $20$ vertices.

\small
\begin{table}[h]
\begin{center}
\begin{tabular}{lclll}
\toprule
Name of $\Gamma$&Order&$\Aut(\Gamma)$&$\QBan(\Gamma)$&Intersection array\\[.5ex]
\midrule
Octahedron $J(4,2)$ (\cite{BanBic})&6&$Z_2 \wr S_3$&$\Z_2\wr_* S_3$&\{4,1;1,4\}\\
Cube $Q_3$ (\cite{BanBic})&8&$S_4 \times \Z_2$&$S_4^+ \times \Z_2$&\{3,2,1;1,2,3\}\\
Paley graph $P_9$ (\cite{BanBic})&9&$S_3\wr \Z_2$&$\Aut(\Gamma)$&\{4,2;1,2\}\\
Petersen graph (\cite{QAutPetersen})&10&$S_5$&$\Aut(\Gamma)$&\{3,2;1,1\}\\
Icosahedron&12&$A_5\times \Z_2$&$\Aut(\Gamma)$&\{5,2,1;1,2,5\}\\
Paley graph $P_{13}$ (\cite{ArthurQSymm})&13&$\Z_{13}\rtimes \Z_6$&$\Aut(\Gamma)$&\{6,3;1,3\}\\
Heawood graph&14&$PGL(2,7)$&$\Aut(\Gamma)$&\{3,2,2;1,1,3\}\\
co-Heawood graph&14&$PGL(2,7)$&$\Aut(\Gamma)$&\{4,3,2;1,2,4\}\\
Line graph of Petersen graph&15&$S_5$&$\Aut(\Gamma)$&\{4,2,1;1,1,4\}\\
Cube $Q_4$ (\cite{Hyperoctahedral})&16&$H_4$&$O_4^{-1}$&\{4,3,2,1;1,2,3,4\}\\
$4\times 4$ rook's graph $H(2,4)$&16&$S_4 \wr \Z_2$&?(has qsym)&\{6,3;1,2\}\\
Shrikhande graph&16&$\Z_4^2 \rtimes D_6$&$\Aut(\Gamma)$&\{6,3;1,2\}\\
Clebsch graph (\cite{foldedcube})&16&$\Z_2^4 \rtimes S_5$&$SO_5^{-1}$&\{5,4;1,2\}\\
Paley graph $P_{17}$ (\cite{ArthurQSymm})&17&$\Z_{17}\rtimes \Z_8$&$\Aut(\Gamma)$&\{8,4;1,4\}\\
Pappus graph&18&ord 216&$\Aut(\Gamma)$&\{3,2,2,1;1,1,2,3\}\\
Johnson graph $J(6,3)$&20&$S_6\times \Z_2$&?&\{9,4,1;1,4,9\}\\
Desargues graph&20&$S_5\times \Z_2$&$\Aut(\Gamma)$&\{3,2,2,1,1;1,1,2,2,3\}\\
Dodecahedron&20&$A_5\times \Z_2$&$\Aut(\Gamma)$&\{3,2,1,1,1;1,1,1,2,3\}\\
Hoffman-Singleton graph&50&$PSU(3,5^2)$&$\Aut(\Gamma)$&\{7,6;1,1\}\\
$K_n$ (\cite{BanBic})&$n$&$S_n$&$S_n^+$&\{$n-1$;1\}\\
$C_n$, $n \neq 4$ (\cite{BanBic})&$n$&$D_n$&$\Aut(\Gamma)$&$(\star^1)$\\
$K_{n,n}$ (\cite{BanBic})&$2n$&$S_n\wr \Z_2$&$S_n^+\wr_* \Z_2$&$\{n,n;1,n\}$\\
$(K_n \square K_2)^{c}$ (\cite{BanBic})&$2n$&$S_n \times \Z_2$&$S_n^+ \times \Z_2$&$(\star^2)$\\
Johnson graph $J(n,2)$,$n\geq 5$&$\binom{n}{2}$&$S_n$&$\Aut(\Gamma)$&\{$2n-4$,$n-3$;1,4\}\\
Kneser graph $K(n,2)$,$n\geq 5$&$\binom{n}{2}$&$S_n$&$\Aut(\Gamma)$&$(\star^3)$\\
Odd graphs $O_k$&$\binom{2k-1}{k-1}$&$S_{2k-1}$&$\Aut(\Gamma)$&$(\star^4)$\\
Hamming graphs $H(n,3)$&$3^n$&$S_3 \wr S_n$&$\Aut(\Gamma)$&$(\star^5)$\\
\bottomrule
\end{tabular}
\end{center}
\vspace{0.1cm}
\caption{Quantum automorphism groups of distance-regular graphs up to 20 vertices and some additional graphs.}\label{table2}
\end{table}
\normalsize
Here 
\begin{itemize}
\item[$(\star^1)$]$=\{2,1,...,1;1,...,1,2\}$ for $n$ even, $\{2,1,...,1;1,...,1,1\}$ for $n$ odd, 
\item[$(\star^2)$]$=\{n-1,n-2, 1;1,n-1,n-2\}$, 
\item[$(\star^3)$]$= \{(n-2)(n-3)/2,2n-8;1,(n-3)(n-4)/2\}$, 
\item[$(\star^4)$]$=\{k,k-1,k-1 \dots, l+1,l+1,l;1,1,2,2,\dots, l,l\}$ for $k =2l-1$,\\\hphantom{=} $\{k,k-1,k-1,\dots l+1,l+1;1,1,2,2,\dots, l-1,l-1,l\}$ for $k=2l$, 
\item[$(\star^5)$]$= \{2n,2n-2,\dots,2;1,2,\dots,n\}$.
\end{itemize}

In \cite{QAutPetersen}, we only needed the values in the intersection array of the Petersen graph to show that this graph has no quantum symmetry. Thus one might guess that the intersection array contains all information about the quantum symmetry of a graph. But this is not the case because of the following. The Shrikhande graph and the $4 \times 4$ rook's graph have the same intersection array. We see that the $4 \times 4$ rook's graph has quantum symmetry by Proposition \ref{Hamm}, whereas the Shrikhande graph has no quantum symmetry, see Subsection \ref{nosym}. This shows that in general one needs to use further graph properties to decide whether or not a graph has quantum symmetry, the intersection array is not enough.

Looking at Table \ref{table2}, we have the following questions. We know that the $4 \times 4$ rook's graph has quantum symmetry, but what is the corresponding quantum group? Does the graph $J(6,3)$, or more generally the graph $J(n,k)$ for $k \geq 3$, have quantum symmetry? Regarding Table \ref{table1}, it is known that there is only one cubic distance-regular graph that is not distance-transitive. This is the Tutte $12$-cage. We do not know whether or not this graph has quantum symmetry and leave it as an open question.
 
\section{Preliminaries}

\subsection{Finite graphs}
The definitions in this subsection are well-known and can for example be found in the books \cite{BCN}, \cite{Petersen}.

Let $\Gamma =(V,E)$ be a \emph{finite graph without multiple edges}, i.e. finite sets of \emph{vertices} $V$ and \emph{edges} $E \subset V \times V$. A graph is called \emph{undirected}, if for all $(i,j) \in E$, we also have $(j,i) \in E$. The \emph{order} of a graph denotes the number of elements in $V$, i.e. the number of vertices in the graph. 

For the rest of this article, we assume that $\Gamma$ is undirected. Let $v \in V$. The vertex $u \in V$ is called a \emph{neighbor} of $v$, if $(v,u) \in E$. A \emph{path} of length $m$ joining two vertices $i,k \in V$ is a sequence of vertices $a_0, a_1, \dots , a_m$ with $i=a_0$, $a_m=k$ such that $(a_n, a_{n+1}) \in E$ for $0\leq n \leq m-1$. A \emph{cycle} of length $m$ is a path of length $m$ where $a_0=a_m$ and all other vertices in the sequence are distinct. The \emph{degree} $\mathrm{deg}\, v$ of a vertex $v \in V$ denotes the number of edges in $\Gamma$ incident with $v$. We say that a graph $\Gamma$ is $k$-\emph{regular} for some $k \in \N_0$, if $\mathrm{deg}\, v = k$ for all $v \in V$. The $3$-regular graphs are also called \emph{cubic} graphs. 

\begin{defn}\label{girth}
Let $\Gamma$ be an undirected graph. We define the \emph{girth} $g(\Gamma)$ of a graph to be the length of a smallest cycle it contains.
\end{defn}

For example, we have $g(\mathrm{P}) =5$, where $\mathrm{P}$ denotes the Petersen Graph.

\begin{defn}\label{inset}
Let $\Gamma$ be an undirected graph. A \emph{clique} is a subset of vertices $W_1 \subseteq V$ such that any vertices are adjacent. A clique, such that there is no clique with more vertices is called \emph{maximal clique}. The \emph{clique number} of $\Gamma$ is the number of vertices of a maximal clique. On the other hand, an \emph{independent set} is a subset $W_2 \subseteq V$ such that no vertices are adjacent. 
\end{defn}

\begin{defn}
Let $\Gamma = (V,E)$ be a $k$-regular graph on $n$ vertices. We say that $\Gamma$ is \emph{strongly regular} if there exist $\lambda$, $\mu \in \N_0$ such that
\begin{itemize}
\item[$(i)$] adjacent vertices have $\lambda$ common neighbors,
\item[$(ii)$] non-adjacent vertices have $\mu$ common neighbors.
\end{itemize}
In this case, we say that $\Gamma$ has parameters $(n,k,\lambda,\mu)$.
\end{defn}

\begin{defn}
Let $\Gamma$ be an undirected graph and $v, w \in V$. 
\begin{itemize}
\item[(a)] The \emph{distance} $d(v,w)$ of two vertices is the length of a shortest path connecting $v$ and $w$.  
\item[(b)] The \emph{diameter} of $\Gamma$ is the greatest distance between any two vertices $v,w$. 
\end{itemize}
\end{defn}

\begin{defn}
Let $\Gamma$ be a regular graph. We say that $\Gamma$ is \emph{distance-regular} if for two vertices $v$, $w$, the number of vertices at distance $k$ to $v$ and at distance $l$ to $w$ only depend on $k$, $l$ and $d(v,w)$. 
\end{defn}

\begin{rem}
The strongly regular graphs are exactly the distance-regular graphs with diameter two. 
\end{rem}

The next definition introduces the intersection array. The intersection array is important to understand the structure of a distance-regular graph. 

\begin{defn}\label{interarray}
Let $\Gamma$ be a distance-regular graph with diameter $d$. The \emph{intersection array} of $\Gamma$ is a sequence of integers $\{b_0,b_1,\dots, b_{d-1}; c_1, c_2, \dots, c_d\}$, such that for any two vertices $v,w$ at distance $d(v,w) = i$, there are exactly $b_i$ neighbors of $w$ at distance $i+1$ to $v$ and exactly $c_i$ neighbors of $w$ at distance $i-1$ to $v$. 
\end{defn}

We will now give the definition of distance-transitive graphs. Besides the \linebreak Shrikhande graph, all graphs appearing in this article are distance-transitive. 

\begin{defn}\label{distance-transitive}
Let $\Gamma$ be a regular graph. We say that $\Gamma$ is \emph{distance-transitive} if for all $(i,k), (j,l) \in V \times V$ with $d(i,k)=d(j,l)$, there is an automorphism $\phi \in \Aut(\Gamma)$ with $\phi(i)=j, \phi(k)=l$. 
\end{defn}

\begin{rem}
Let $\Gamma$ be a distance-transitive graph and let $v,w \in V$. Since we have an automorphism $\phi \in \Aut(\Gamma)$ with $\phi(v)=x, \phi(w)=y$ for every pair of vertices $x,y$ with $d(x,y)=d(v,w)$, we see that the number of vertices at distance $k$ to $v$ and at distance $l$ to $w$ only depend on $k$, $l$ and $d(v,w)$. Thus, we see that every distance-transitive graph is distance-regular. 
\end{rem}

We recall the definition of line graphs and incidence graphs since those constructions will be used explicitely in this article. 

\begin{defn}
Let $\Gamma$ be a finite, undirected graph. The \emph{line graph} $L(\Gamma)$ of $\Gamma$ is the graph whose vertices correspond to edges of $\Gamma$ and whose vertices are connected if and only if the corresponding edges are incident in $\Gamma$ (i.e. are connected by a vertex). 
\end{defn}

See for example \cite{Coxeter} for the next definition. 

\begin{defn}
Given $c,d \in \N$, a \emph{configuration} $(P,L)$ consists of points $P=\{p_1, \dots, p_a\}$ and lines $L=\{L_1,\dots, L_b\}$ in a plane, such that
\begin{itemize}
\item[(i)] there are $c$ points on each line and $d$ lines through each point,
\item[(ii)] two different lines intersect each other at most once,
\item[(iii)] two different points are connected by one line at most.
\end{itemize}
\end{defn}

\begin{defn}
Let $(P,L)$ be a configuration with points $P=\{p_1, \dots, p_a\}$ and lines $L=\{L_1,\dots, L_b\}$. Then the \emph{incidence graph} of the configuration is a bipartite graph consisting of vertices $P \cup L$. Here $P$ and $L$ are independent sets (in the sense of Definition \ref{inset}) and we have an edge between $p_j$ and $L_k$ if and only if $p_j$ is adjacent to $L_k$ in the configuration. 
\end{defn}

The following definitions concern the automorphism group of a graph. This will be generalized in the next subsection. We start with the definition of the adjacency matrix.

\begin{defn}
Let $\Gamma$ be a finite graph of order $n$, without multiple edges. The \emph{adjacency matrix} $\varepsilon \in M_n(\{0,1\})$ is the matrix where $\varepsilon_{ij} = 1$ if $(i,j) \in E$ and $\varepsilon_{ij} = 0$ otherwise. 
\end{defn}

\begin{defn}
Let $\Gamma=(V,E)$ be a finite graph without multiple edges. A \emph{graph automorphism} is a bijection $\sigma: V \to V$ such that $(i,j) \in E$ if and only if $(\sigma(i), \sigma(j)) \in E$. The set of all graph automorphisms of $\Gamma$ forms a group, the \emph{automorphism group} $\Aut(\Gamma)$. If $\Gamma$ has $n$ vertices, we can view $\Aut(\Gamma)$ as a subgroup of the symmetric group $S_n$, in the following way.
\begin{align*}
\Aut(\Gamma) = \{ \sigma \in S_n\;|\; \sigma \varepsilon = \varepsilon \sigma\}\subset S_n.
\end{align*} 
Here $\varepsilon$ denotes the adjacency matrix of the graph. 
\end{defn}

\subsection{Quantum automorphism groups of finite graphs}
We start with the definition of compact matrix quantum groups. Those were defined by Woronowicz \cite{CMQG1, CMQG2} in 1987.

\begin{defn}
A \emph{compact matrix quantum group} $G$ is a pair $(C(G),u)$, where $C(G)$ is a unital (not necessarily commutative) $C^*$-algebra which is generated by $u_{ij}$, $1 \leq i,j \leq n$, the entries of a matrix $u \in M_n(C(G))$. Moreover, the *-homomorphism $\Delta: C(G) \to C(G) \otimes C(G)$, $u_{ij} \mapsto \sum_{k=1}^n u_{ik} \otimes u_{kj}$ must exist, and $u$ and its transpose $u^{t}$ must be invertible. 
\end{defn}

The \emph{quantum symmetric group} $S_n^+$ defined by Wang \cite{WanSn} is the quantum analogue of the symmetric group $S_n$. It is the compact matrix quantum group, where
\begin{align*}
C(S_n^+) := C^*(u_{ij}, \, 1 \leq i,j \leq n \, | \, u_{ij} = u_{ij}^* = u_{ij}^2, \, \sum_{l} u_{il} = \sum_{l} u_{li} =1).
\end{align*} 

Now, we are ready to define quantum automorphism groups of finite graphs. Those are quantum subgroups of $S_n^+$. The following definition was given by Banica \cite{QBan} in 2005.

\begin{defn}
Let $\Gamma = (V, E)$ be a finite graph on $n$ vertices $V = \{1, ... , n \}$. The \emph{quantum automorphism group} $\QBan(\Gamma)$ is the compact matrix quantum group $(C(\QBan(\Gamma)), u)$, where $C(\QBan(\Gamma))$ is the universal $C^*$-algebra with generators $u_{ij}$, $1 \leq i,j \leq n$ and relations
\begin{align}\allowdisplaybreaks
&u_{ij} = u_{ij}^*= u_{ij}^2, &&1 \leq i,j,k \leq n,\label{QA1}\\ 
&\sum_{l=1}^n u_{il} = 1 = \sum_{l=1}^n u_{li}, &&1 \leq i \leq n,\label{QA2}\\
&u_{ji}u_{lk} = u_{lk}u_{ji} = 0, && (i,k) \notin E, (j,l) \in E,\label{QA3}\\
&u_{ij}u_{kl} = u_{kl}u_{ij} = 0, && (i,k) \notin E, (j,l) \in E,\label{QA4}
\end{align} 
where \eqref{QA3} and \eqref{QA4} are equivalent to $u\varepsilon = \varepsilon u$. 
\end{defn}

There is another definition of quantum automorphism groups by Bichon \cite{QBic} in 2003. This is a quantum subgroup of the one defined by Banica. 

\begin{defn}
Let $\Gamma = (V, E)$ be a finite graph on $n$ vertices $V = \{1, ... , n \}$. The \emph{quantum automorphism group} $\QBic(\Gamma)$ is the compact matrix quantum group $(C(\QBic(\Gamma)), u)$, where $C(\QBic(\Gamma))$ is the universal $C^*$-algebra with generators $u_{ij}$, $1 \leq i,j \leq n$, Relations \eqref{QA1} -- \eqref{QA4} and
\begin{align}\allowdisplaybreaks
&u_{ij}u_{kl} = u_{kl}u_{ij}, &&(i,k), (j,l) \in E.\label{QA5}
\end{align} 
\end{defn}

The next definiton was given by Banica and Bichon in \cite{BanBic}. 

\begin{defn}
Let $\Gamma = (V,E)$ be a finite graph. We say that $\Gamma$ has \emph{no quantum symmetry} if $C(\QBan(\Gamma))$ is commutative, or equivalently
\begin{align*}
\QBan(\Gamma) = \Aut(\Gamma).
\end{align*}
\end{defn} 

\section{Tools for proving commutativity of the generators}\label{tools}
In this section, we develop tools to obtain commutation relations between the generators of the quantum automorphism group. The following, well-known fact can be found for example in \cite{QAutPetersen}. 

\begin{lem}\label{com}
Let $(u_{ij})_{1 \leq i,j \leq n}$ be the generators of $C(\QBan(\Gamma))$. If we have 
\begin{align*}
u_{ij}u_{kl} = u_{ij}u_{kl}u_{ij}
\end{align*}
 then $u_{ij}$ and $u_{kl}$ commute.
\end{lem}

\begin{proof}
Since $u_{ij}u_{kl}u_{ij}$ is selfadjoint, we infer the result. 
\end{proof}

The next lemma yields that, if we want to show that a graph has no quantum symmetry, it suffices to look at words $u_{ij}u_{kl}$, where $d(i,k)=d(j,l)$. 

\begin{lem}\label{dist}
Let $\Gamma$ be a finite, undirected graph and let $(u_{ij})_{1 \leq i,j \leq n}$ be the generators of $C(\QBan(\Gamma))$. If we have $d(i,k) \neq d(j,l)$, then $u_{ij}u_{kl} = 0$.
\end{lem}

\begin{proof}
We may assume $m:=d(i,k) < d(j,l)$. For $m=0$, we get $i=k$ and hence $u_{ij}u_{il} =0$ since $u_{ij}, u_{il}$ are orthogonal projections by Relations \eqref{QA1}, \eqref{QA2}. If $m =1$, then Relation \eqref{QA3} yields the assertion. Otherwise, there is a path of length $m\geq2$ from $i$ to $k$, say $i, a_1, a_2, \dots, a_{m-1},k$. By Relation \eqref{QA2}, we get\allowdisplaybreaks
\begin{align*}
u_{ij} u_{kl} &= u_{ij}\left(\sum_{b_1} u_{a_1b_1}\right)\left(\sum_{b_2} u_{a_2b_2}\right)\dots \left(\sum_{b_{m-1}} u_{a_{m-1}b_{m-1}}\right)u_{kl}\\
&=\sum_{b_1, \dots, b_{m-1}} u_{ij}u_{a_1b_1} u_{a_2b_2}\dots u_{a_{m-1}b_{m-1}}u_{kl}.
\end{align*}
Since $d(j,l) > m$, there is no path of length $m$ between $j$ and $l$. Thus, for all $b_0:=j, b_1, \dots, b_{m-1}, b_m :=l$, there are two vertices $b_x$, $b_{x+1}$ with $(b_x, b_{x+1} )\notin E$. We get $u_{a_xb_x}u_{a_{x+1}b_{x+1}} =0$ by Relation \eqref{QA3} and therefore
\begin{align*}
 u_{ij}u_{a_1b_1} u_{a_2b_2}\dots u_{a_{m-1}b_{m-1}}u_{kl} = 0
\end{align*}
for all $b_1, \dots, b_{m-1}$. We conclude
\begin{align*}
u_{ij} u_{kl} =\sum_{b_1,\dots ,b_{m-1}} u_{ij}u_{a_1b_1} u_{a_2b_2}\dots u_{a_{m-1}b_{m-1}}u_{kl}=0.
\end{align*}
\end{proof}

The following lemma concerns distance-transitive graphs. It allows us to work with one pair of vertices $(j_1, l_1)$ in distance $m$ to obtain commutativity of all $u_{ij}$, $u_{kl}$ with $d(i,k)=d(j,l)=m$. 

\begin{lem}\label{disttrans}
Let $\Gamma$ be a distance-transitive graph and let $(u_{ij})_{1 \leq i,j \leq n}$ be the generators of $C(\QBan(\Gamma))$. Let $j_1, l_1 \in V$ and put $m:=d(j_1,l_1)$. If $u_{aj_1}u_{bl_1} =u_{bl_1}u_{aj_1}$ for all $a,b$ with $d(a,b)=m$, then we have $u_{ij}u_{kl} =u_{kl}u_{ij}$ for all $i,k, j,l$ with $d(j,l)=m =d(i,k)$.
\end{lem}

\begin{proof}
Let $j_1,l_1 \in V$ and $u_{aj_1}u_{bl_1} =u_{bl_1}u_{aj_1}$ for all $a,b$ with $d(a,b)=m$. The map $\tilde{\phi}: C(\QBan(\Gamma)) \to C(\QBan(\Gamma)), u_{ij} \mapsto (\phi u\phi^{-1})_{ij}=u_{\phi(i)\phi(j)}$ is a *-isomorphism for all $\phi \in \Aut(\Gamma)$, because we have $\epsilon(\phi u\phi^{-1})=(\phi u\phi^{-1})\epsilon$ since $\phi$ and $\phi^{-1}$ commute with the adjacency matrix $\epsilon$ of $\Gamma$. For all pairs $j,l$ with $d(j,l)=m$, there is a graph automorphism $\phi_{j,l}$ with $\phi_{j,l}(j_1)=j$, $\phi_{j,l}(l_1)=l$, since $\Gamma$ is distance-transitive (see Definition \ref{distance-transitive}). Let $\tilde{\phi}_{j,l}$ be the $*$-isomorphism corresponding to $\phi_{j,l}$. We obtain
\begin{align*}
u_{ij}u_{kl} = \tilde{\phi}_{j,l}(u_{\phi_{j,l}^{-1}(i),j_1}u_{\phi_{j,l}^{-1}(k),l_1})=\tilde{\phi}_{j,l}(u_{\phi_{j,l}^{-1}(k),l_1}u_{\phi_{j,l}^{-1}(i),j_1}) =u_{kl}u_{ij},
\end{align*}
for all $i,j,k,l$ with $d(j,l)=m =d(i,k)$,  since we know $d(\phi_{j,l}^{-1}(i), \phi_{j,l}^{-1}(k))=m$ and thus have $u_{\phi_{j,l}^{-1}(i),j_1}u_{\phi_{j,l}^{-1}(k),l_1}=u_{\phi_{j,l}^{-1}(k),l_1}u_{\phi_{j,l}^{-1}(i),j_1}$ by assumption. 
\end{proof}

For the rest of this section, we give criteria on properties of the graph $\Gamma$ (for example the girth or the intersection array) that allow us to say that certain generators of $\QBan(\Gamma)$ commute. The following theorem generalizes Theorem 3.2 of \cite{QAutPetersen}. Recall that the girth $g(\Gamma)$ of a graph was defined in Definition \ref{girth}. 

\begin{thm}\label{g5}
Let $\Gamma$ be an undirected graph with girth $g(\Gamma) \geq5$. Then $\QBan(\Gamma) = \QBic(\Gamma)$.
\end{thm}

\begin{proof}
Let $(u_{ij})_{1 \leq i,j \leq n}$ be the generators of $C(\QBan(\Gamma))$ and let $(i,k) \in E$, $(j,l) \in E$. It holds
\begin{align*}
u_{ij}u_{kl} = u_{ij}u_{kl}\left(\sum_{s;(l,s) \in E} u_{is}\right)
 \end{align*} 
by Relations \eqref{QA2} and \eqref{QA3}. 

Take $s \neq j$ with $(l,s) \in E$. Since we also have $(j,l) \in E$, the only common neighbor of $j$ and $s$ is $l$ as otherwise we would get a quadrangle in $\Gamma$, contradicting $g(\Gamma) \geq 5$. Hence, for all $a \neq l$, we have $(a,s) \notin E$ or $(a,j) \notin E$. Then Relation \eqref{QA3} implies $u_{ka}u_{is} =0$ or $u_{ij}u_{ka} = 0$ for all $a \neq l$. By also using Relations \eqref{QA1} and \eqref{QA2}, we get 
\begin{align*}
u_{ij}u_{kl}u_{is} = u_{ij}\left(\sum_{a=1}^n u_{ka}\right)u_{is}
=u_{ij}u_{is}
=0.
\end{align*} 
Therefore, we obtain
\begin{align*}
u_{ij}u_{kl} = u_{ij}u_{kl}\left(\sum_{s;(l,s) \in E} u_{is} \right) = u_{ij}u_{kl}u_{ij}
\end{align*}
and we conclude $u_{ij}u_{kl} = u_{kl}u_{ij}$ by Lemma \ref{com}. This yields $\QBan(\Gamma) = \QBic(\Gamma)$.
\end{proof}

The upcoming lemma deals with graphs where adjacent vertices have one common neighbor. 

\begin{lem}\label{oneneighbor}
Let $\Gamma$ be an undirected graph such that adjacent vertices have exactly one common neighbor. Then $\QBan(\Gamma) = \QBic(\Gamma)$. In particular,  we have $\QBan(\Gamma) = \QBic(\Gamma)$ for distance-regular graphs with $b_0= b_1+2$ in the intersection array.
\end{lem}

\begin{proof}
Let $(u_{ij})_{1 \leq i,j \leq n}$ be the generators of $C(\QBan(\Gamma))$ and let $(i,k), (j,l) \in E$. Using Relations \eqref{QA2}, \eqref{QA3} we get 
\begin{align*}
u_{ij} u_{kl} = u_{ij}u_{kl} \left(\sum_{p;(l,p) \in E} u_{ip}\right).
\end{align*}
There is exactly one $p_1\neq j, (l,p_1) \in E$ such that $(p_1, j) \in E$ since adjacent vertices have exactly one neighbor and we have $(j,l) \in E$. In this case we have $(j,a) \notin E$ or $(a,p_1)\notin E$ for $a \neq l$, because $l$ is the only common neighbor of $j$ and $p_1$. We deduce
\begin{align*}
u_{ij} u_{kl} u_{ip_1} = u_{ij}\left( \sum_{a} u_{ka}\right) u_{ip_1} = u_{ij} u_{ip_1} =0
\end{align*}
by Relations \eqref{QA3} and \eqref{QA2}.

Now, let $p \notin\{j, p_1\}$, $(l,p)\in E$ and let $s$ be the only common neighbor of $i$ and $k$. It holds 
\begin{align*}
u_{ij}u_{kl}u_{ip} = u_{ij} \left(\sum_{a} u_{sa}\right) u_{kl} u_{ip} = u_{ij} u_{sp_1}u_{kl}u_{ip}
\end{align*}
by Relations  \eqref{QA2}, \eqref{QA3}, since $p_1$ is the only common neighbor of $j$ and $l$. 
We also know that $j$ is the only common neighbor of $p_1$ and $l$ and since we have $(l,p)\in E$, we deduce $(p_1, p) \notin E$. Relations \eqref{QA2}, \eqref{QA4} now yield
\begin{align*}
0 = u_{sp_1}u_{ip} = u_{sp_1}\left(\sum_{a} u_{al}\right)u_{ip} = u_{sp_1}u_{kl}u_{ip}
\end{align*}
because $k$ is the only common neighbor of $s$ and $i$. Thus we have
\begin{align*}
u_{ij}u_{kl}u_{ip}= u_{ij}u_{sp_1}u_{kl}u_{ip} = 0. 
\end{align*}
Summarising, we get
\begin{align*}
u_{ij} u_{kl} = u_{ij} u_{kl} u_{ij}
\end{align*}
and by Lemma \ref{com} we conclude $u_{ij} u_{kl} = u_{kl} u_{ij} \text{ for } (i,k), (j,l) \in E$.
\end{proof}

The next lemma is technical and mostly used to shorten the upcoming proofs. 

\begin{lem}\label{abk}
Let $\Gamma$ be a finite, undirected graph and let $(u_{ij})_{1 \leq i,j \leq n}$ be the generators of $C(\QBan(\Gamma))$. Let $d(i,k)=d(j,l)=m$. Let $q$ be a vertex with $d(j,q) =s$, $d(q,l)=t$ and $u_{kl}u_{aq}=u_{aq}u_{kl}$ for all $a$ with $d(a,k)=t$. Then 
\begin{align*}
u_{ij}u_{kl} = u_{ij}u_{kl}\sum_{\substack{p; d(l,p)=m,\\ d(p,q)=s}} u_{ip}.
\end{align*}
In particular, if we have $m=2$ and if $\QBan(\Gamma) = \QBic(\Gamma)$ holds, then choosing $s=t=1$ implies
\begin{align*}
u_{ij}u_{kl} = u_{ij}u_{kl}\sum_{\substack{p; d(l,p)=2,\\ (p,q) \in E}} u_{ip}.
\end{align*}
\end{lem}

\begin{proof}
Using Lemma \ref{dist} and Relation \eqref{QA2}, we know
\begin{align*}
u_{ij}u_{kl}=u_{ij}u_{kl}\sum_{p;d(l,p)=m} u_{ip}.
\end{align*}
Additionally, we want to obtain
\begin{align*}
u_{ij}u_{kl}=u_{ij}u_{kl}\sum_{\substack{p;d(l,p)=m,\\d(p,q)=s}} u_{ip}.
\end{align*}

 By Relation \eqref{QA2} and Lemma \ref{dist}, we have 
 \begin{align*}
 u_{ij}u_{kl} = u_{ij}\left(\sum_a u_{aq}\right)u_{kl} = u_{ij}\left(\sum_{\substack{a;d(i,a)=s,\\d(k,a)=t}} u_{aq}\right) u_{kl}.
 \end{align*}
Furthermore, also by Relation \eqref{QA2} and Lemma \ref{dist}, it holds \allowdisplaybreaks
 \begin{align*}
 u_{ij}u_{kl} &= u_{ij}\left(\sum_{\substack{a;d(i,a)=s,\\d(k,a)=t}} u_{aq}\right) u_{kl}\\
&= u_{ij}\left(\sum_{\substack{a;d(i,a)=s,\\d(k,a)=t}} u_{aq}\right) u_{kl} \sum_{p} u_{ip} \\
&=u_{ij}\left(\sum_{\substack{a;d(i,a)=s,\\d(k,a)=t}} u_{aq}\right) u_{kl} \sum_{p;d(l,p)=m} u_{ip}.
 \end{align*}
 Since we have $u_{kl}u_{aq}=u_{aq}u_{kl}$ for all $a$ with $d(a,k)=t$ and by Lemma \ref{dist}, we get
 \begin{align*}
  u_{ij}u_{kl} &=u_{ij}\left(\sum_{\substack{a;d(i,a)=s,\\d(k,a)=t}} u_{aq}\right) u_{kl} \sum_{p;d(l,p)=m} u_{ip} \\
  &=u_{ij}u_{kl}\left(\sum_{\substack{a;d(i,a)=s,\\d(k,a)=t}} u_{aq}\right) \sum_{p;d(l,p)=m} u_{ip}\\
  &=u_{ij}u_{kl}\left(\sum_{\substack{a;d(i,a)=s,\\d(k,a)=t}} u_{aq}\right) \sum_{\substack{p;d(l,p)=m,\\d(p,q)=s}} u_{ip}.
 \end{align*}
 Using $u_{kl}u_{aq}=u_{aq}u_{kl}$ for all $a$ with $d(a,k)=t$ again, we obtain
 \begin{align*}
  u_{ij}u_{kl} &=u_{ij}u_{kl}\left(\sum_{\substack{a;d(i,a)=s,\\d(k,a)=t}} u_{aq}\right) \sum_{\substack{p;d(l,p)=m,\\d(p,q)=s}} u_{ip}\\
 &=u_{ij}\left(\sum_{\substack{a;d(i,a)=s,\\d(k,a)=t}} u_{aq}\right)u_{kl} \sum_{\substack{p;d(l,p)=m,\\d(p,q)=s}} u_{ip}.
 \end{align*}
 By Lemma \ref{dist} and Relation \eqref{QA2}, we get
 \begin{align*}
u_{ij}u_{kl} &=u_{ij}\left(\sum_{a} u_{aq}\right)u_{kl} \sum_{\substack{p;d(l,p)=m,\\d(p,q)=s}} u_{ip}\\
  &=u_{ij}u_{kl} \sum_{\substack{p;d(l,p)=m,\\d(p,q)=s}} u_{ip}
 \end{align*}
 and this completes the proof.
\end{proof}

The following result is helpful, if one has a specific labelling of the vertices and it is not to hard to see which vertices are in distance $m$ to the given ones. 

\begin{lem}\label{krit}
Let $\Gamma$ be a finite, undirected graph and let $(u_{ij})_{1 \leq i,j \leq n}$ be the generators of $C(\QBan(\Gamma))$. Let $d(i,k)=d(j,l)=m$ and let $p\neq j$ be a vertex with $d(p,l)=m$. Let $q$ be a vertex with $d(q,l)=s$ and $d(j,q) \neq d(q,p)$. Then
\begin{align*}
u_{ij}\left(\sum_{\substack{t;d(t,j)=d(t,p)=m,\\d(t,q)=s}} u_{kt}\right)u_{ip} =0.
\end{align*}
Especially, if $l$ is the only vertex satisfying $d(l,q)=s$, $d(l,j)=m$ and $d(l,p)=m$, we obtain $u_{ij}u_{kl}u_{ip} = 0$.
\end{lem}

\begin{proof}
By Relation \eqref{QA2} and Lemma \ref{dist}, it holds
\begin{align*}
u_{ij}\left(\sum_{\substack{t;d(t,j)=d(t,p)=m,\\d(t,q)=s}} u_{kt}\right)u_{ip} = u_{ij}\left(\sum_{\substack{a;d(i,a)=d(j,q),\\d(k,a)=s}} u_{aq}\right)\left(\sum_{\substack{t;d(t,j)=d(t,p)=m,\\d(t,q)=s}} u_{kt}\right)u_{ip}.
\end{align*}

Now, let $b\neq t$, for all $t$ appearing in the above sum. We prove
\begin{align*}
u_{ij}\left(\sum_{\substack{a;d(i,a)=d(j,q),\\d(k,a)=s}} u_{aq}\right)u_{kb}u_{ip} =0.
\end{align*}
Indeed, if $d(b,q) \neq s$ or $d(b,p) \neq m$, then we get
\begin{align*}
u_{ij}\left(\sum_{\substack{a;d(i,a)=d(j,q),\\d(k,a)=s}} u_{aq}\right)u_{kb}u_{ip} =0,
\end{align*}
by Lemma \ref{dist}. On the other hand, if $d(b,q) = s$ and $d(b,p) = m$, then we have $d(b,j) \neq m$ by assumption. This yields
\begin{align*}
u_{ij}\left(\sum_{\substack{a;d(i,a)=d(j,q),\\d(k,a)=s}} u_{aq}\right)u_{kb}u_{ip} =u_{ij}u_{kb}u_{ip} =0,
\end{align*}
also by Relation \eqref{QA2} and Lemma \ref{dist}. Using $\sum_b u_{kb}=1$, we deduce 
\begin{align*}
u_{ij}\left(\sum_{\substack{t;d(t,j)=d(t,p)=m,\\d(t,q)=s}} u_{kt}\right)u_{ip} &= u_{ij}\left(\sum_{\substack{a;d(i,a)=d(j,q),\\d(k,a)=s}} u_{aq}\right)\left(\sum_{\substack{t;d(t,j)=d(t,p)=m,\\d(t,q)=s}} u_{kt}\right)u_{ip}\\ 
&= u_{ij}\left(\sum_{\substack{a;d(i,a)=d(j,q),\\d(k,a)=s}} u_{aq}\right)\left(\sum_{b} u_{kb}\right)u_{ip}\\ 
&=u_{ij}\left(\sum_{\substack{a;d(i,a)=d(j,q),\\d(k,a)=s}} u_{aq}\right)u_{ip}.
\end{align*}
Since we assumed $d(j,q) \neq d(q,p)$, the condition $d(i,a)=d(j,q)$ implies $d(i,a) \neq d(q,p)$. Thus, Lemma \ref{dist} yields $u_{aq}u_{ip} =0$ for all such $a$ and we get
\begin{align*}
u_{ij}u_{kl}u_{ip} =u_{ij}\left(\sum_{\substack{a;d(i,a)=d(j,q),\\d(k,a)=s}} u_{aq}\right)u_{ip}=0.
\end{align*}
\end{proof}

We see that certain values in the intersection array of a graph $\Gamma$ give commutation relations of the generators of $\QBan(\Gamma)$ in the subsequent lemma. Recall Definition \ref{interarray} of the intersection array.

\begin{lem}\label{ia}
Let $\Gamma$ be a distance-regular graph with intersection array \linebreak$\{b_0,b_1,\dots, b_{d-1}; c_1, c_2, \dots, c_d\}$ and let $(u_{ij})_{1 \leq i,j \leq n}$ be the generators of $C(\QBan(\Gamma))$. Let $c_m \geq 2$ for some $m \geq 2$ and assume 
\begin{align*}
u_{ij}u_{kl}=u_{kl}u_{ij}
\end{align*}
for all vertices $i,j,k,l$ with $d(i,k)=d(j,l)=m-1$. If
\begin{itemize}
\item[(a)] $c_2 =1$ and $b_1+1=b_0$,
\item[(b)] $c_2=1$ and $b_1+2=b_0$,
\item[(c)] or $c_2 =2$, $m=2$ and $b_1+3=b_0$,
\end{itemize}
then we have $u_{ij}u_{kl}=u_{kl}u_{ij}$ for all $i,j,k,l$ with $d(i,k)=d(j,k)=m$.
\end{lem}

\begin{proof}
Let $d(i,k) = d(j,l) = m$. Since $c_m \geq 2$, there are two neighbors $t, \tau$ of $j$ in distance $m-1$ to $l$. Since we have $u_{ac}u_{bd} = u_{bd}u_{ac}$ for $d(a,b) = d(c,d) = m-1$ by assumption, we get 
\begin{align*}
u_{ij} u_{kl}= u_{ij}u_{kl}\left(\sum_{\substack{p; d(p,l)=m,\\(t,p) \in E}} u_{ip} \right) \quad \text{ and } \quad u_{ij}u_{kl} = u_{ij}u_{kl}\left(\sum_{\substack{p; d(p,l)=m,\\(\tau,p) \in E}} u_{ip} \right)
\end{align*}
by Lemma \ref{abk}.
We deduce\allowdisplaybreaks
\begin{align*}
u_{ij}u_{kl} &= u_{ij}u_{kl}\left(\sum_{\substack{p; d(p,l)=m,\\(t,p) \in E}} u_{ip} \right)\\
&= u_{ij}u_{kl}\left(\sum_{\substack{p; d(p,l)=m,\\(\tau,p) \in E}} u_{ip} \right)\left(\sum_{\substack{p; d(p,l)=m,\\(t,p) \in E}} u_{ip} \right)\\
&=u_{ij}u_{kl}\left(\sum_{\substack{p; d(p,l)=m,\\(\tau,p) \in E,(t,p) \in E}} u_{ip} \right).
\end{align*}

In case (a), we know from $b_1+1=b_0$ that $\Gamma$ does not contain a triangle. Therefore we have $d(t,\tau) =2$, since they have a common neighbor $j$ and they are not connected, because otherwise there would be a triangle in $\Gamma$. Then $c_2=1$ implies that $j$ is the only common neighbor of $t$ and $\tau$. Thus only $j$ satisfies $d(j,l)=m$, $(\tau,j)\in E$, $(t,j) \in E$. 

In case (b), we either have $(t, \tau) \in E$ or $d(t,\tau)=2 $. If $(t, \tau) \in E$, then $b_1+2=b_0$ implies that $j$ is the only common neighbor of $t$ and $\tau$. If $d(t,\tau)=2$, we get that $j$ is the only common neighbor of $t$ and $\tau$ because $c_2=1$. 

In case (c), we get that $j$ and $l$ are the only common neighbors of $t$ and $\tau$ by similar considerations as in case (b). Thus, $j$ is the only vertex satisfying the above conditions. 

Summarizing, in all three cases
\begin{align*}
u_{ij}u_{kl}=u_{ij}u_{kl}\left(\sum_{\substack{p; d(p,l)=m,\\(\tau,p) \in E,(t,p) \in E}} u_{ip} \right) = u_{ij}u_{kl}u_{ij}
\end{align*}
and then Lemma \ref{com} completes the proof. 
\end{proof}

\section{Families of distance-transitive graphs}\label{families}

In this section we deal with families of distance-transitive graphs. The considered families are well-known and can for example be found in \cite{BCN}. The strategy of proving that a graph has no quantum symmetry is the following.

\begin{itemize}
\item[$(1)$] By Lemma \ref{dist}, we know that it suffices to show $u_{ij}u_{kl}=u_{kl}u_{ij}$ for $d(i,k)=d(j,l)$. 
\item[$(2)$] Choose a distance $d(i,k)=d(j,l)=m$ (usually, one starts with $m=1$, then $m=2$ and so on). 
\item[$(3)$] First check if Theorem \ref{g5}, Lemma \ref{oneneighbor} or Lemma \ref{ia} applies. If this is the case, then we know $u_{ij}u_{kl}=u_{kl}u_{ij}$ for $d(i,k)=d(j,l)=m$. 
\item[$(4)$] Otherwise, using Lemma \ref{disttrans}, we know that it is enough to show $u_{ij_1}u_{kl_1}=u_{kl_1}u_{ij_1}$ for one pair $(j_1,l_1)$ and all $(i,k)$ with $d(i,k)=m$. 
\item[$(5)$] If we know the neighbors of $l_1$ (for example because of a known construction of the graph), one can apply Lemma \ref{krit} and use the equations to deduce $u_{ij}u_{kl}=u_{kl}u_{ij}$ for $d(i,k)=d(j,l)=m$. 
\item[$(6)$] If this does not work, we have to treat this distance in the graph as a special case and try to get $u_{ij}u_{kl}=u_{kl}u_{ij}$ by other methods.
\end{itemize}

\subsection{The odd graphs}\label{odd}
We show that the odd graphs do not have quantum symmetry. Here we use that strategy as described above.  

\begin{defn}
Let $k \geq 2$. The graph $O_k$ with vertices corresponding to $(k-1)$-subsets of $\{1, \dots, 2k-1\}$, where two vertices are connected if and only if the corresponding subsets are disjoint is called \emph{odd graph}. 
\end{defn}

The odd graphs have the following properties, see for example \cite[Proposition 9.1.7]{BCN}. 

\begin{rem}
Odd graphs are distance-transitive with $\Aut(O_k)=S_{2k-1}$, diameter $k-1$ and girth $g(O_k)\geq 5$ for $k \geq 3$. They have the intersection array 
\begin{align*}
&\{k,k-1,k-1 \dots, l+1,l+1,l;1,1,2,2,\dots, l,l\} &&\text{ for } k =2l-1,\\
&\{k,k-1,k-1,\dots, l+1,l+1;1,1,2,2,\dots, l-1,l-1,l\} &&\text{ for } k=2l.
\end{align*} 
\end{rem}

\begin{ex}
The graph $O_2$ is the triangle and $O_3$ is the Petersen graph. We know that those graphs have no quantum symmetry by \cite{WanSn} and \cite{QAutPetersen}, respectively.
\end{ex}

We are now ready to prove Theorem \ref{thm1.1}.

\begin{thm}
The odd graphs have no quantum symmetry.
\end{thm}

\begin{proof}
Since we know that $O_2$ has no quantum symmetry, we can assume $k\geq 3$. Then we know that $O_k$ has girth $g(O_k) \geq 5$ and thus we get $\QBan(O_k)=\QBic(O_k)$ by Theorem \ref{g5}, i.e. Relation \eqref{QA5} holds. 

Take $d(i,p) = d(j,q)=2$. We want to show $u_{ij}u_{pq}=u_{pq}u_{ij}$. Since $O_k$ is distance-transitive, it is enough to show $u_{ij}u_{pq}=u_{pq}u_{ij}$ for $j=\{1,\dots, k-1\}$, $q=\{1,\dots, k-2, k\}$ by Lemma \ref{disttrans}.\\

\noindent\emph{Step 1: It holds $u_{ij}u_{pq}=u_{ij}u_{pq}\sum\limits_{\substack{s=1,\\s\neq k-1}}^k u_{id_s}$, where $d_s$ is defined as $\{1, \dots, k\}\backslash \{s\}$.}\\
The only common neighbor of $j$ and $q$ is $a=\{k+1,\dots, 2k-1\}$. Since $\QBan(O_k)=\QBic(O_k)$, we get
\begin{align*}
u_{ij}u_{pq} =u_{ij}u_{pq}\sum_{\substack{b;d(b,q)=2,\\(a,b) \in E}} u_{ib}
\end{align*}
by Lemma \ref{abk}. Neighbors of $a=\{k+1,\dots, 2k-1\}$ are $d_s=\{1, \dots, k\}\backslash \{s\}$, where $s=1,\dots, k$. Those are in distance two to $q$ if $s\neq {k-1}$ and we have $d_k=j$. Thus
\begin{align*}
u_{ij}u_{pq} =u_{ij}u_{pq}\sum_{\substack{b;d(b,q)=2,\\(a,b) \in E}} u_{ib}=u_{ij}u_{pq}\sum_{\substack{s=1,\\s\neq k-1}}^k u_{id_s}.
\end{align*}

\noindent\emph{Step 2: It holds $u_{ij}u_{pq}u_{id_s}=0$ for all $s \in \{1,\dots, k-2\}$.}\\
Take $d_s$ with $s \in \{1,\dots, k-2\}$. Let $t =\{1,\dots ,k-2, k+1\}$. We get that $d(j,t)=2$ and $d(q,t)=2$ since they have the common neighbor $\{k,k+2,\dots, 2k-1\}$ and $\{k-1, k+2, \dots 2k-1\}$, respectively. Because $t \cup d_s = \{1, \dots, k+1\}$, we see that there is no $(k-1)$-subset of $\{1,\dots, 2k-1\}$ disjoint to $t$ and $d_s$ and we deduce $d(t,d_s) \neq 2$. Furthermore, we get that $q$ and $r_s= \{1, \dots, k-1, k+1\}\backslash \{s\}$ are the only vertices in distance two to $j,d_s$ and $t$. This holds since the only $(k-1)$-subsets of $\{1,\dots, 2k-1\}$ that have $k-2$ elements in common with $\{1,\dots, k-1\}$, $\{1,\dots, k\}\backslash\{s\}$, $s\neq k-1,k$ and $\{1,\dots ,k-2, k+1\}$ are $\{1,\dots, k-2, k\}$ and  $\{1, \dots, k-1, k+1\}\backslash \{s\}$. Now, Lemma \ref{krit} yields 
\begin{align}\label{eq1}
u_{ij}(u_{pq} + u_{pr_s})u_{id_s} =0.
\end{align}
Since we have $g(O_k) \geq 5$, $i$ and $p$ have exactly one common neighbor which we denote by $c$. Recall that $a$ is the only common neighbor of $j$ and $q$. Using  Equation \eqref{eq1}, we get
\begin{align*}
u_{ca}u_{ij}(u_{pq} + u_{pr_s})u_{id_s} =0
\end{align*}
and because of Relation $\eqref{QA5}$, we obtain
\begin{align*}
u_{ij}u_{ca}(u_{pq} + u_{pr_s})u_{id_s} =0.
\end{align*}
Since the sets $\{k+1,\dots, 2k-1\}$ and $\{1, \dots k-1, k+1\}\backslash \{s\}$ are not disjoint, we have $(a,r_s) \notin E$. But we know $(c,p) \in E$ by the choice of $c$, thus we get $u_{ij}u_{ca}u_{pr_s}u_{id_s}=0$ by Relation \eqref{QA3}. This yields 
\begin{align}\label{eq2}
u_{ij}u_{ca}u_{pq}u_{id_s} =0.
\end{align}
The vertex $a$ is the only common neighbor of $j$ and $p$, therefore it holds
\begin{align*}
u_{ij}u_{ca}u_{pq} = u_{ij}\left(\sum_{e} u_{ce}\right)u_{pq}=u_{ij}u_{pq} 
\end{align*}
by Relations \eqref{QA2} and \eqref{QA3}. We deduce 
\begin{align*}
u_{ij}u_{pq}u_{id_s}=u_{ij}u_{ca}u_{pq}u_{id_s} =0
\end{align*}
from Equation \eqref{eq2}.\\

\noindent\emph{Step 3: It holds $u_{ij}u_{pq}=u_{pq}u_{ij}$.}\\
Recall that $d_k =j$. From previous steps, we get 
\begin{align*}
u_{ij}u_{pq}= u_{ij}u_{pq}\sum_{\substack{s=1, \\ s \neq k-1}}^k u_{id_s} = u_{ij}u_{pq}u_{ij}
\end{align*}
 and Lemma \ref{com} yields $u_{ij}u_{pq}=u_{pq}u_{ij}$.

The odd graph $O_k$ has diameter $k-1$. Taking this and Lemma \ref{dist} into account, it remains to show $u_{ij}u_{pq}=u_{pq}u_{ij}$ for $3 \leq d(i,p)=d(j,q)\leq k-1$. We have $c_2=1$, $b_1+1=b_0$ and $c_d\geq 2$ for all $d\geq3$ in the intersection array of $O_k$ and thus we obtain the desired equations by using Lemma \ref{ia} (a) $(k-3)$-times. 
\end{proof}

\subsection{Hamming graphs}\label{Hamming}
In this subsection, we give a precise description for which values $d,q \in \N$ the Hamming graph $H(d,q)$ has quantum symmetry and for which it does not. 

\begin{defn}
Let $S=\{1,\dots, q\}$ for $q \in \N$ and let $d\in \N$. The \emph{Hamming graph} $H(d,q)$ is the graph with vertex set $S^d$, where vertices are adjacent if they differ in exactly one coordinate.
\end{defn}

We state some properties of the Hamming graphs in the following remark, see for example \cite[Theorem 9.2.1]{BCN} and \cite[Subsection 12.4.1]{spectra}.
\begin{rem}
The Hamming graphs are distance-transitive and we have $H(d,q) = K_q^{\square d}$, where $\square$ denotes the Cartesian product of graphs. Here, for two graphs $X=(V_1,E_1)$, $Y=(V_2,E_2)$, the Cartesian product $X\square Y$ is the graph with vertex set $V_1 \times V_2$, where two vertices $(u_1,u_2)$, $(v_1,v_2)$ are connected if and only if either $u_1=v_1$ and $(u_2,v_2)\in E_2$ or $(u_1,v_1)\in E_1$ and $u_2=v_2$.
\end{rem}

Hamming graphs include the following families of graphs.
\vspace{0.2cm}
\begin{ex}~
\begin{itemize}
\item[(i)]The Hamming graphs $H(d,1)$ are the complete graphs $K_d$. 
\item[(ii)] For $q=2$, we obtain the hypercube graphs.
\item[(iii)] The Hamming graphs $H(2,q)$ are the $q \times q$ rook's graphs. 
\end{itemize}
\end{ex}

The following proposition is an easy consequence of \cite[Proposition 4.1]{BanBic}.

\begin{prop}\label{Hamm}
Let $q \geq 4$, $d \in \N$ or $q=2$, $d \geq 2$. Then $H(d,q)$ has quantum symmetry. 
\end{prop}

\begin{proof}
Let $q \geq 4$, $d \in \N$. We know $H(d,q) = K_q^{\square d}$ and by Proposition $4.1$ of \cite{BanBic}, we have a surjective *-homomorphism $\phi: C(\QBan(H(d,q)) \to C(S_q^+) \otimes C(\QBan(K_q^{\square d-1}))$. Thus, if $q \geq 4$, this yields that $C(\QBan(H(d,q)))$ is non-commutative.  

Let $q=2$, $d \geq 2$. We get a surjective *-homomorphism $\phi: C(\QBan(H(d,2)) \to C(H_2^+) \otimes C(\QBan(K_2^{\square d-2}))$ by Proposition $4.1$ of \cite{BanBic}, thus $C(\QBan(H(d,2))$ is non-commutative. 
\end{proof}

The following theorem completes the proof of Theorem \ref{thm1.2}.

\begin{thm}
The Hamming graphs $H(d,3)$ do not have quantum symmetry, for $d \in \N$.
\end{thm}

\begin{proof}
Let $i$,$j$ be adjacent vertices. Thus they differ in exactly one coordinate $i_s \neq j_s$. Since we have $q=3$, this means that there is only one vertex that differs in exactly one coordinate to $i$ and $j$, namely $k$ with $k_a=i_a=j_a$ for all $a \neq s$ and $k_s \neq i_s$, $k_s \neq j_s$. Therefore, adjacent vertices have exactly one neighbor and we get $\QBan(H(d,3))=\QBic(H(d,3))$ by Lemma \ref{oneneighbor}. Hence Relation $\eqref{QA5}$ holds.

The Hamming graph $H(d,3)$ has diameter $d$. Using Lemma \ref{dist}, it remains to show $u_{ij}u_{kl} = u_{kl}u_{ij}$ for all $i,j,k,l$ with $2 \leq d(i,k)=d(j,l) \leq d$ to obtain $\QBan(H(d,3))= \Aut(H(d,3))$. For this, consider $s=(1,\dots, 1)$, $t^{(m)}=(t^{(m)}_1, \dots t^{(m)}_d)$, where $t^{(m)}_{1}=\dots=t^{(m)}_{m}=2$, $t^{(m)}_{m+1}=\dots =t^{(m)}_{d}=1$ for $2\leq m \leq d$, $p_1 =(2,1,\dots, 1)$ and $p_2=(1,2,1,\dots,1)$.\\

\noindent\emph{Step 1: The only common neighbor of $p_1$ and $p_2$ in distance $m$ to $t^{(m)}$ is $s$.}\\
The only common neighbors of $p_1$ and $p_2$ are $s$ and $(2,2,1,\dots, 1)$. We obtain that $s$ is the only common neighbor of $p_1, p_2$ in distance $m$ to $t^{(m)}$, since $d(t,(2,2,1, \dots, 1))=m-2$.\\

\noindent\emph{Step 2: We have $u_{ij}u_{kl}=u_{kl}u_{ij}$ for $d(i,k)=d(j,l)=2$.}\\
Let $d(i,k)=d(j,l)=2$. By Lemma \ref{disttrans}, it is enough to consider $j=s$ and $l=t^{(2)}$. Since we know $(j,p_1), (j,p_2) \in E$, $(l,p_1),(l,p_2) \in E$ and have $\QBan(H(d,3))=\QBic(H(d,3))$, we get
\begin{align*}
u_{ij}u_{kl}= u_{ij}u_{kl}\sum_{\substack{q;d(q,l)=2,\\ \hphantom{q;}(q,p_1)\in E}}u_{iq}
\quad \text{  and  } \quad 
u_{ij}u_{kl}= u_{ij}u_{kl}\sum_{\substack{q;d(q,l)=2,\\ \hphantom{q;}(q,p_2)\in E}}u_{iq}
\end{align*}
by Lemma \ref{abk}. We deduce
\begin{align*}
u_{ij}u_{kl}= u_{ij}u_{kl}\sum_{\substack{q;d(q,l)=2,\\ \hphantom{q;}(q,p_1)\in E, (q,p_2)\in E}}u_{iq}.
\end{align*}
 By \emph{Step 1}, we know that $j$ is the only common neighbor of $p_1$, $p_2$ at distance two to $l$. Therefore we obtain $u_{ij}u_{kl}=u_{ij}u_{kl}u_{ij}$. Then Lemma \ref{com} yields $u_{ij}u_{kl}=u_{kl}u_{ij}$. \\

\noindent\emph{Step 3: We have $u_{ij}u_{kl}=u_{kl}u_{ij}$ for $d(i,k)=d(j,l)=3$.}\\
Now, let $d(i,k)=d(j,l)=3$. By Lemma \ref{disttrans}, we can choose $j=s$ and $l=t^{(3)}$. Since we know $(j,p_1), (j,p_2) \in E$, $d(l,p_1)=(l,p_2) =2$ and have $u_{ac}u_{bd}=u_{bd}u_{ac}$ for all $a,b,c,d$ with $d(a,b)=d(c,d)=2$ by \emph{Step 2}, we get
\begin{align*}
u_{ij}u_{kl}= u_{ij}u_{kl}\sum_{\substack{q;d(q,l)=3,\\ \hphantom{q;}(q,p_1)\in E}}u_{iq}
\quad \text{  and  } \quad 
u_{ij}u_{kl}= u_{ij}u_{kl}\sum_{\substack{q;d(q,l)=3,\\ \hphantom{q;}(q,p_2)\in E}}u_{iq}
\end{align*}
by using Lemma \ref{abk}. We obtain
\begin{align*}
u_{ij}u_{kl}= u_{ij}u_{kl}\sum_{\substack{q;d(q,l)=3,\\ \hphantom{q;}(q,p_1)\in E, (q,p_2)\in E}}u_{iq}
\end{align*}
and get $u_{ij}u_{kl}=u_{ij}u_{kl}u_{ij}$, since the only common neighbor of $p_1$, $p_2$ at distance three to $l$ is $j$ by \emph{Step 1}. Then Lemma \ref{com} yields $u_{ij}u_{kl}=u_{kl}u_{ij}$. 

Repeating this argument $(d-3)$-times yields the assertion. 
\end{proof}

\subsection{The Johnson graphs $J(n,2)$ and the Kneser graphs $K(n,2)$}

In the following we show that $J(n,2)$ and $K(n,2)$ have no quantum symmetry for $n \geq 5$. For $n \leq 5$, the quantum automorphism groups of $J(n,2)$ and $K(n,2)$ are already known from \cite{BanBic}, \cite{QAutPetersen}. The natural question here is what happens for $k>2$. Since we know that the odd graphs $O_k$ are the Kneser graphs $K(2k-1,k-1)$, we dealt with some special case in Subsection \ref{odd}. 

\begin{defn}
Let $n, k \in \N$. 
\begin{itemize}
\item[(i)]The \emph{Johnson graph} $J(n,k)$ is the graph with vertices corresponding to $k$-subsets of $\{1, \dots, n\}$, where two vertices are connected if and only if the intersection of the corresponding subsets has $(k-1)$ elements. 
\item[(ii)] The \emph{Kneser graph} $K(n,k)$ is the graph with vertices corresponding to $k$-subsets of $\{1, \dots, n\}$, where two vertices are connected if and only if the corresponding subsets are disjoint.
\end{itemize}
\end{defn}
\vspace{0.1cm}
\begin{ex}~
\begin{itemize} 
\item[(i)] The Kneser graphs $K(n,1)$ are the complete graphs $K_n$.
\item[(ii)]The Johnson graphs $J(n,2)$ are the line graphs of the complete graphs $K_n$. 
\end{itemize}
\end{ex}

\begin{rem}
The Johnson graphs $K(n,2)$ are distance-transitive with diameter $2$, see \cite[Theorem 9.1.2]{BCN}. For $n\leq 5$, the quantum automorphism groups of $K(n,2)$ are known, since $K(4,2) = 3K_2$, $K(5,2) =P$, where $P$ denotes the Petersen graph. 
\end{rem}

The following gives a proof of Theorem \ref{thm1.3}.

\begin{thm}
For $n \geq 5$, the Johnson graph $J(n,2)$ and the Kneser graph $K(n,2)$ do not have quantum symmetry. 
\end{thm}

\begin{proof}
We show that $J(n,2)$ has no quantum symmetry for $n \geq 5$. This suffices because $K(n,2)$ is the complement of $J(n,2)$. 

Let $(i,k),(j,l) \in E$. We want to prove $u_{ij}u_{kl}=u_{kl}u_{ij}$. Since $J(n,2)$ is distance-transitive, it suffices to show this for $j=\{1,2\}$ and $l=\{1,3\}$ by Lemma \ref{disttrans}.\\

\noindent\emph{Step 1: We have $u_{ij}u_{kl} = u_{ij}u_{kl}\left(\sum\limits_{\substack{a=2,\\a \neq 3}}^n u_{i\{1,a\}} + \sum\limits_{\substack{b=2,\\b \neq 3}}^n u_{i\{3,b\}}\right)$.}\\
By Relations \eqref{QA2}, \eqref{QA3}, it holds
\begin{align*}
u_{ij}u_{kl} = u_{ij}u_{kl}\sum_{p; (l,p) \in E} u_{ip}.
\end{align*}
Since $l=\{1,3\}$, it has neighbors $\{1,a\}$, $\{3,b\}$,  $a, b \neq 1,3$ and thus
\begin{align*}
u_{ij}u_{kl} = u_{ij}u_{kl}\sum_{p; (l,p) \in E} u_{ip} = u_{ij}u_{kl}\left(\sum_{\substack{a=2,\\a \neq 3}}^n u_{i\{1,a\}} + \sum_{\substack{b=2,\\b \neq 3}}^n u_{i\{3,b\}}\right).
\end{align*}

\noindent\emph{Step 2: It holds $u_{ij}u_{kl}u_{i\{1,a\}}=0$ for $a \in \{4,\dots, n\}$ and $u_{ij}u_{kl}u_{i\{2,3\}}=0$.}\\
Let $p=\{1,a\}$. The common neighbors of $p$ and $j = \{1,2\}$ are $\{2,a\}$ and $\{1,c\}$ for $c \notin \{1,2,a\}$. Therefore
\begin{align}\label{1}
u_{ij}\left(u_{k\{2,a\}} + \sum_{\substack{c=3,\\ c\neq a}}^n u_{k\{1,c\}}\right)u_{ip}= u_{ij}\left(\sum_{a;(a,j)\in E, (a,p) \in E} u_{ka}\right) u_{ip}=u_{ij}u_{ip} =0.
\end{align}
The only common neighbors of $j,p$ and $\{2,d\}$, where $d \notin \{1,2,a\}$, are $\{2,a\}$ and $\{1,d\}$. We also know $d(j,\{2,d\})=1\neq 2=d(p,\{2,d\})$ and thus we obtain 
\begin{align}\label{2}
u_{ij}(u_{k\{2,a\}} + u_{k\{1,d\}})u_{ip} =0
\end{align}
for all such $d$ by Lemma \ref{krit}. 
This yields
\begin{align*}
u_{ij}(u_{k\{2,a\}} + u_{k\{1,d\}})u_{ip} =0 =u_{ij}(u_{k\{2,a\}} + u_{k\{1,3\}})u_{ip}
\end{align*}
and we deduce
\begin{align*}
u_{ij}u_{k\{1,d\}}u_{ip}= u_{ij}u_{k\{1,3\}}u_{ip}
\end{align*}
for $d \notin \{1,2,a\}$. Putting this into Equation \eqref{1}, we infer
\begin{align*}
u_{ij}(u_{k\{2,a\}} + (n-3)u_{k\{1,3\}})u_{ip}=0.
\end{align*}
Using Equation \eqref{2} with $d=3$, we get
\begin{align*}
(n-4)u_{ij}u_{k\{1,3\}}u_{ip}=0.
\end{align*}
Since we assumed $n \geq 5$, we obtain $u_{ij}u_{k\{1,3\}}u_{ip}=u_{ij}u_{kl}u_{ip}=0$. Furthermore, we also get $u_{ij}u_{k\{2,a\}}u_{i\{1,a\}}=0$ by Equation \eqref{2}. By repeating the arguments for $p=\{2,3\}$, one obtains $u_{ij}u_{k\{1,3\}}u_{i\{2,3\}}=u_{ij}u_{kl}u_{i\{2,3\}}=0$.\\

\noindent\emph{Step 3: It holds $u_{ij}u_{kl}u_{i\{3,b\}}=0$ for $b \in \{4,\dots, n\}$.}\\
Let $p=\{3,b\}$, $b \in \{4,\dots, n\}$. Since $l=\{1,3\}$ and $\{1,b\}$ are the only common neighbors of $j=\{1,2\}$, $p$ and $\{1,e\}$, where $e \notin \{1,2,3,b\}$, we have
\begin{align}\label{3}
u_{ij}(u_{kl} + u_{k\{1,b\}})u_{ip}=0
\end{align} 
by Lemma \ref{krit}, because $(j,\{1,e\})\in E, (p,\{1,e\})\notin E$. Now, multiplying Equation \eqref{3} by $u_{ip}u_{kl}$ from the left, we obtain
\begin{align*}
u_{ip}u_{kl}u_{ij}(u_{kl} + u_{k\{1,b\}})u_{ip}=u_{ip}u_{kl}u_{ij}u_{kl}u_{ip} +u_{ip}u_{kl}u_{ij}u_{k\{1,b\}}u_{ip}=0.
\end{align*} 
Similar to $u_{ij}u_{kl}u_{i\{1,a\}}=0$ (see Step \emph{2}), we obtain $u_{kl}u_{ij}u_{k\{1,b\}}=0$. Thus, we get 
\begin{align*}
u_{ip}u_{kl}u_{ij}u_{kl}u_{ip}=0,
\end{align*} 
which implies $u_{ij}u_{kl}u_{ip}=0$.\\ 

\noindent\emph{Step 4: We have $u_{ij}u_{kl}=u_{kl}u_{ij}$ for $(i,k),(j,l) \in E$.}\\
From \emph{Steps 1--3}, we get
\begin{align*}
u_{ij}u_{kl} = u_{ij}u_{kl}\left(\sum\limits_{\substack{a=2,\\a \neq 3}}^n u_{i\{1,a\}} + \sum\limits_{\substack{b=2,\\b \neq 3}}^n u_{i\{3,b\}}\right) = u_{ij}u_{kl}u_{ij}
\end{align*}
and therefore obtain $u_{ij}u_{kl}=u_{kl}u_{ij}$ for $(i,k),(j,l) \in E$ by Lemma \ref{com}.\\

\noindent\emph{Step 5: We have $u_{ij}u_{kl}=u_{kl}u_{ij}$ for $d(i,k)=d(j,l)=2$.}\\
Let $d(i,k)=2=d(j,l)$. We show $u_{ij}u_{kl}=u_{kl}u_{ij}$, where we can choose $j = \{1,2\}, l=\{3,4\}$ by Lemma \ref{disttrans}. 
The vertices $\{1,3\}$, $\{2,4\}$ are common neighbors of $j$ and $l$. By Lemma \ref{abk}, we get
\begin{align*}
u_{ij}u_{kl} = u_{ij}u_{kl} \sum_{\substack{\{a,b\};\{a,b\} \cap \{3,4\}= \emptyset,\\ \hphantom{\{a,b\}}(\{a,b\}, \{1,3\}) \in E}} u_{i \{a,b\}} =u_{ij}u_{kl} \sum_{\substack{\{a,b\};\{a,b\} \cap \{3,4\}= \emptyset,\\ \hphantom{\{a,b\}}1 \in \{a,b\}}} u_{i \{a,b\}}
\end{align*}
and 
\begin{align*}
u_{ij}u_{kl} = u_{ij}u_{kl} \sum_{\substack{\{c,d\};\{c,d\} \cap \{3,4\}= \emptyset,\\ \hphantom{\{a,b\}}(\{c,d\}, \{2,4\}) \in E}} u_{i \{c,d\}} =u_{ij}u_{kl} \sum_{\substack{\{c,d\};\{c,d\} \cap \{3,4\}= \emptyset,\\ \hphantom{\{a,b\}}2 \in \{c,d\}}} u_{i \{c,d\}}.
\end{align*} 
We deduce 
\begin{align*}
u_{ij}u_{kl} = u_{ij}u_{kl} \left(\sum_{\substack{\{a,b\};\{a,b\} \cap \{3,4\}= \emptyset,\\ \hphantom{\{a,b\}}1 \in \{a,b\}}} u_{i \{a,b\}}\right)\left(\sum_{\substack{\{c,d\};\{c,d\} \cap \{3,4\}= \emptyset,\\ \hphantom{\{c,d\}}2 \in \{c,d\}}} u_{i \{c,d\}}\right).
\end{align*}
Since we know $u_{i\{a,b\}} u_{i\{c,d\}} =0$ for $\{a,b\} \neq \{c,d\}$, we obtain
\begin{align*}
u_{ij}u_{kl} &= u_{ij}u_{kl} \left(\sum_{\substack{\{a,b\};\{a,b\} \cap \{3,4\}= \emptyset,\\ \hphantom{\{a,b\}}1 \in \{a,b\}}} u_{i \{a,b\}}\right)\left(\sum_{\substack{\{c,d\};\{c,d\} \cap \{3,4\}= \emptyset,\\ \hphantom{\{c,d\}}2 \in \{c,d\}}} u_{i \{c,d\}}\right)\\
 &= u_{ij}u_{kl} \left(\sum_{\substack{\{a,b\};\{a,b\} \cap \{3,4\}= \emptyset,\\ \hphantom{\{a,b\}}1 \in \{a,b\}, 2 \in \{a,b\}}} u_{i \{a,b\}}\right)\\
 &=u_{ij}u_{kl}u_{ij},
\end{align*}
since $\{1,2\}$ is the only subset containing $1$ and $2$. Then Lemma \ref{com} completes the proof, since $J(n,2)$ has diameter $2$.
\end{proof}

\subsection{Moore graphs of diameter two}\label{moore}

We show that the Moore graphs of diameter two have no quantum symmetry. Those are precisely the strongly regular graphs with girth five or equivalently all strongly regular graphs with $\mu =0$, $\lambda=1$. There are only three such graphs, the $5$-cycle, the Petersen graph and the Hoffman-Singleton graph, where the existence of a fourth graph with parameters $(3250,57,0,1)$ is still an open problem, see for example \cite[Section 6.7]{BCN}.

\begin{thm}
Strongly regular graphs with girth five have no quantum symmetry. 
\end{thm}

\begin{proof}
Since the graph $\Gamma$ has girth five, we get $\QBan(\Gamma) = \QBic(\Gamma)$ by Theorem \ref{g5}. 

Let $d(i,k)=d(j,l)=2$. It remains to show $u_{ij}u_{kl}=u_{kl}u_{ij}$.\\

\noindent\emph{Step 1: We have $u_{ij}u_{kl}=u_{ij}u_{kl}\sum\limits_{\substack{p;d(p,l)=2,\\(p,t)\in E}} u_{ip}$, where $t$ is the only common neighbor of $j$ and $l$.}\\
There exist exactly one $s \in E$ such that $(i,s) \in E, (k,s) \in E$ and exactly one $t \in E$ such that $(j,t) \in E, (l,t) \in E$, since otherwise we would get a quadrangle in $\Gamma$. We get
\begin{align*}
u_{ij}u_{kl} = u_{ij}u_{st}u_{kl},
\end{align*}
by Relations \eqref{QA2}, \eqref{QA3} and it holds 
\begin{align*}
 u_{ij}u_{kl}= u_{ij}u_{kl}\sum_{\substack{p;d(p,l)=2,\\(p,t)\in E}} u_{ip},
\end{align*}
because we have $\QBan(\Gamma) = \QBic(\Gamma)$ and thus can use Lemma \ref{abk}. \\

\noindent\emph{Step 2: We have $u_{ij}u_{kl}u_{ip} =0$ for $p\neq j$ with $d(p,l)=2$, $(p,t)\in E$.}\\
If $\Gamma$ is $2$-regular (the $5$-cycle) then we are done, because the only vertex $p$ with $d(p,l)=2$, $(p,t)\in E$ is $j$. Therefore, we can assume that $\Gamma$ is $n$-regular with $n \geq 3$ in the remaining part of the proof. Take $p \neq j$ with $d(p,l)=2, (p,t)\in E$. It holds 
\begin{align*}
u_{ij}u_{kl}u_{ip} = u_{ij}u_{kl}u_{st}u_{ip} = u_{ij} \left(\sum_{\substack{a;(a,k)\in E,\\d(a,i)=2}} u_{ab}\right)u_{kl}u_{st}u_{ip}, 
\end{align*}
where we choose $b\neq t$ with $(b,l)\in E$, which implies $d(b,j)=d(b,p)=2$. We know that $l$ is the only common neighbor of $b$ and $t$, because $\Gamma$ has girth five. We deduce 
\begin{align*}
u_{ij} \left(\sum_{\substack{a;(a,k)\in E,\\d(a,i)=2}} u_{ab}\right)u_{kl}u_{st}u_{ip} =u_{ij} \left(\sum_{\substack{a;(a,k)\in E,\\d(a,i)=2}} u_{ab}\right)u_{st}u_{ip} 
\end{align*}
by Relations \eqref{QA2}, \eqref{QA3}. Furthermore, there exist exactly one $e \in E$ such that $(i,e) \in E, (e,a) \in E$ for all $a$ with $d(a,i)=2$ and exactly one $f \in E$ such that $(j,f) \in E, (b,f) \in E$, since otherwise we would get a quadrangle in $\Gamma$. This yields
\begin{align*}
u_{ij}u_{ab}u_{st}u_{ip} = u_{ij}u_{ef}u_{ab}u_{st}u_{ip} = u_{ij}u_{ef}u_{ab}u_{ip}u_{st}, 
\end{align*}
where we also used $u_{ip}u_{st} =u_{st}u_{ip}$. Because of $u_{ef}u_{ab}=u_{ab}u_{ef}$, we get 
\begin{align*}
u_{ij}u_{ab}u_{st}u_{ip}= u_{ij}u_{ef}u_{ab}u_{ip}u_{st} = u_{ij}u_{ab}u_{ef}u_{ip}u_{st}.
\end{align*}
It holds $(f,p)\notin E$, because otherwise $j$ and $p$ would have two common neighbors, $t$ and $f$, where we know $t\neq f$ since we have $(b,f) \in E$ whereas $d(b,t) =2$. 
Because we know $(i,e) \in E$, we obtain
\begin{align*}
u_{ij}u_{ab}u_{st}u_{ip} =u_{ij}u_{ab}u_{ef}u_{ip}u_{st}=0,
\end{align*}
by Relation \eqref{QA3}. Summarizing, we get 
\begin{align*}
u_{ij}u_{kl}u_{ip} = u_{ij} \left(\sum_{\substack{a;(a,k)\in E,\\d(a,i)=2}} u_{ab}\right)u_{st}u_{ip} =\sum_{\substack{a;(a,k)\in E,\\d(a,i)=2}} u_{ij} u_{ab}u_{st}u_{ip}= 0.
\end{align*}

\noindent\emph{Step 3: We have $u_{ij}u_{kl}=u_{kl}u_{ij}$.}\\
By the previous steps, we conclude
\begin{align*}
u_{ij}u_{kl}= u_{ij}u_{kl}\sum\limits_{\substack{p;d(p,l)=2\\(p,t)\in E}} = u_{ij}u_{kl}u_{ij},
\end{align*}
which implies that $u_{ij}$ and $u_{kl}$ commute by Lemma \ref{com}. This completes the proof.
\end{proof}

\begin{rem}
Since it is known that the $5$-cycle and the Petersen graph have no quantum symmetry, the only new insight we get is that the Hoffman-Singleton graph has no quantum symmetry. Also, if the strongly regular graph with parameters  $(3250, 57, 0,1)$  exists, then it has no quantum symmetry by the previous theorem. 
\end{rem}

\subsection{Paley graphs $P_9$, $P_{13}$ and $P_{17}$}\label{paley}
The Paley graphs are constructed using finite fields. We use this construction to show that $P_9$, $P_{13}$ and $P_{17}$ have no quantum symmetry. 

\begin{defn}
Let $q$ be a prime power with $q = 1 \,\,\mathrm{mod} \,\,4$ and let $\mathbb{F}_q$ be the finite field with $q$ elements. The \emph{Paley graph} $P_q$ is the graph with vertex set $\mathbb{F}_q$, where vertices are connected if and only if their difference is a square in $F$. 
\end{defn}

\begin{rem}
The Paley graphs are distance-transitive. 
\end{rem}

\begin{prop}
The Paley graphs $P_9$, $P_{13}$ and $P_{17}$ have no quantum symmetry. 
\end{prop}

\begin{proof}
Note that the Paley graph $P_9$ is strongly regular with parameters (9,4,1,2). Thus, Lemma \ref{oneneighbor} yields $\QBan(P_9) = \QBic(P_9)$.
Since $P_9$ is self-complementary, the arguments in the proof of Lemma \ref{oneneighbor} also work for $(i,k), (j,l) \notin E$. Thus $C(\QBan(P_9))$ is commutative.

Regarding the Paley graph $P_{13}$, observe that $0,1,3,4,9,10,12$ are the squares in $\mathbb{F}_{13}$. Let $(i,k),(j,l)\in E$. Since $P_{13}$ is distance-transitive, we can choose $j=1$, $l=2$ by Lemma \ref{disttrans}. We get
\begin{align*}
u_{i1}u_{k2} = u_{i1}u_{k2}\sum_{p;(p,2)\in E} u_{ip},
\end{align*}
by Relations \eqref{QA2}, \eqref{QA3}. The neighbors of $2$ are $1,3,5,6,11,12$. The task is now to find for every neighbor $p\neq 1$ of $2$ a vertex $q$, $d(q,2)=s$ with $d(1,q) \neq d(q,p)$, such that $2$ is the only common neighbor of $1,p$ in distance $s$ to $q$, because then we get $u_{ij}u_{kl}u_{ip} =0$ by Lemma \ref{krit}. We find the following vertices that fulfill these properties: $11$ for $3$, $11$ for $5$, $3$ for $6$, $5$ for $11$ and $5$ for $12$. Thus we get $u_{i1}u_{k2} =u_{i1}u_{k2}u_{i1}$ and we conclude $\QBan(P_{13}) = \QBic(P_{13})$ by Lemma \ref{com} and Lemma \ref{disttrans}. Since $P_{13}$ is self-complementary, the same arguments work for $(i,k), (j,l) \notin E$ and we get that $P_{13}$ has no quantum symmetry. 

Concerning the Paley graph $P_{17}$, observe that $0,2,4,8,9,13,15,16$ are the squares in $\mathbb{F}_{17}$. Let $(i,k),(j,l)\in E$. Since $P_{17}$ is distance-transitive, we can choose $j=1$, $l=2$ by Lemma \ref{disttrans}. We get
\begin{align*}
u_{i1}u_{k2} = u_{i1}u_{k2}\sum_{p;(p,2)\in E} u_{ip},
\end{align*}
by Relations \eqref{QA2}, \eqref{QA3}. The neighbors of $2$ are $1,3,4,6,10,11,15,17$. As for $P_{13}$, the task is to find for every neighbors $p\neq 1$ of $2$ a vertex $q$, $d(q,2)=s$ with $d(1,q) \neq d(q,p)$, such that $2$ is the only common neighbor of $1,p$ in distance $s$ to $q$, because then we get $u_{ij}u_{kl}u_{ip} =0$ by Lemma \ref{krit}. We have the following vertices that fulfill these properties: $10$ for $3$, $10$ for $4$, $17$ for $6$, $4$ for $10$, $15$ for $11$, $11$ for $15$, $6$ for $17$. We get $u_{i1}u_{k2} =u_{i1}u_{k2}u_{i1}$ and therefore $\QBan(P_{17}) = \QBic(P_{17})$ by Lemma \ref{com} and Lemma \ref{disttrans}. Since $P_{17}$ is self-complementary, we the same arguments work for $(i,k), (j,l) \notin E$ and we obtain that $P_{17}$ has no quantum symmetry. 
\end{proof}

\begin{rem}
Note that it was already shown in \cite{BanBic} that $P_9$ has no quantum symmetry. In \cite{ArthurQSymm}, it was proven that $P_{13}$ and $P_{17}$ have no quantum symmetry. Thus, we just give alternative proofs of those facts. One could try to get similar results for other Paley graphs $P_{q}$, $q >17$. But using our method one has to treat them case by case, we do not get a general statement for all Paley graphs in this way. 
\end{rem}

\section{Quantum automorphism groups of cubic distance-transitive graphs}\label{cubicdist}

In this section we study the quantum automorphism groups of all cubic distance-transitive graphs. Those quantum automorphism groups are known for the complete graph $K_4$, the complete bipartite graph $K_{3,3}$, the cube $Q_3$ from \cite{BanBic} and for the Petersen graph from \cite{QAutPetersen}. The following result was established by Biggs and Smith in \cite{BiggsSmith}.

\begin{thm}[Biggs, Smith]
There are exactly twelve cubic distance-transitive graphs.
\end{thm}

Thus, there are eight remaining graphs. We treat them case by case. We first start with a useful lemma. The proof is similar to the one of Theorem 3.3 in \cite{QAutPetersen}.

\begin{lem}\label{d2}
Let $\Gamma$ be a cubic graph with girth $g(\Gamma) \geq 5$ and let $(u_{ij})_{1 \leq i,j \leq n}$ be the generators of $C(\QBan(\Gamma))$. Then we have $u_{ij}u_{kl} = u_{kl}u_{ij}$ for $d(i,k) = 2 = d(j,l)$. 
\end{lem}

\begin{proof}
Let $d(i,k) =2 = d(j,l)$. There exist exactly one $s \in E$ such that $(i,s) \in E, (k,s) \in E$ and exactly one $t \in E$ such that $(j,t) \in E, (l,t) \in E$, since otherwise we would get a quadrangle in $\Gamma$. We know $u_{st}u_{kl} =u_{kl}u_{st}$ by Theorem \ref{g5} and therefore we get
\begin{align*}
u_{ij}u_{kl} = u_{ij}u_{kl}\left(\sum_{\substack{p;d(l,p)=2,\\(t,p) \in E}} u_{ip}\right)
\end{align*} 
by Lemma \ref{abk}. The graph $\Gamma$ is $3$-regular and we know that $j$ and $l$ are neighbors of $t$, where $d(l,l)=0\neq 2$. We deduce
\begin{align*}
u_{ij}u_{kl}= u_{ij}u_{kl}(u_{ij} + u_{iq}) 
\end{align*}
where we denote by $q$ the third neighbor of $t$. 

It holds 
\begin{align*}
u_{ij}u_{kl} = u_{ij}\left(\sum_{a}u_{sa}\right)u_{kl} = u_{ij}\left(\sum_{a;(j,a)\in E, (a,l)\in E}u_{sa}\right)u_{kl}=u_{ij}u_{st}u_{kl},
\end{align*}
since $t$ is the only common neighbor of $j$ and $l$.
Observe that
\begin{align*}
u_{ij}u_{st}u_{kq}u_{iq} = 0 \quad \text{ and } \quad u_{ij}u_{st}u_{kj}u_{iq} = u_{ij}u_{kj}u_{st}u_{iq} = 0 
\end{align*}
by Relations \eqref{QA1} and $u_{st}u_{kj} = u_{kj}u_{st}$ from Theorem \ref{g5}. We therefore get
\begin{align*}\allowdisplaybreaks
u_{ij}u_{kl}u_{iq}&=u_{ij}u_{st}u_{kl}u_{iq} \\
&= u_{ij}u_{st}(u_{kl}+u_{kj} + u_{kq})u_{iq}\\ 
&= u_{ij}u_{st}\left(\sum_{a;(t,a) \in E} u_{ka}\right)u_{iq}\\
&= u_{ij}u_{st}\left(\sum_{a=1}^n u_{ka}\right)u_{iq}\\
&= u_{ij}u_{st}u_{iq},
\end{align*}
where we also used Relations \eqref{QA2}, \eqref{QA3}. 
By Relation \eqref{QA1} and using $u_{ij}u_{st} = u_{st}u_{ij}$, we obtain
\begin{align*}
u_{ij}u_{kl}u_{iq} = u_{ij}u_{st}u_{iq} = u_{st}u_{ij}u_{iq} = 0,
\end{align*}
since $j \neq q$. 

We conclude
\begin{align*}
u_{ij}u_{kl} = u_{ij}u_{kl}u_{ij}.
\end{align*}
Then Lemma \ref{com} yields $u_{ij}u_{kl} = u_{kl} u_{ij}$ and this completes the proof.
\end{proof}

\begin{lem}\label{g}
Let $\Gamma$ be a cubic distance-regular graph of order $\geq 10$ and let \linebreak$(u_{ij})_{1 \leq i,j \leq n}$ be the generators of $C(\QBan(\Gamma))$. Then we have $u_{ij}u_{kl}=u_{kl}u_{ij}$ for $d(i,k)=d(j,l) \leq 2$. 
\end{lem}

\begin{proof}
For the intersection array, we have $b_1=2$ and $c_2=1$ for all cubic distance-regular graphs of order $\geq 10$. It follows that all those graphs have girth $\geq 5$, because we need adjacent vertices to have at least one common neighbor ($b_1+1 < k$ for a $k$-regular graph) to get a triangle and vertices in distance two to have at least two common neighbors ($c_2\geq 2$) to get a quadrangle in the graph. Because of Theorem \ref{g5} and Lemma \ref{d2}, we get $u_{ij} u_{kl} = u_{kl} u_{ij}$ for $d(i,k) = d(j,l) \leq 2$.
\end{proof}

In the following we study the quantum automorphism groups of the remaining eight cubic distance-transitive graphs and prove Theorem \ref{mainthm}. As a reminder we write the intersection array in parantheses to the graph. We always write $u_{ij}$, $1 \leq i,j \leq n$, for the generators of $C(\QBan(\Gamma))$. 

\subsection{The Heawood graph $(\{3,2,2;1,1,3\})$}
Since the Heawood graph has diameter three, we have $d(i,k), d(j,l) \leq 3$ for all $i,j,k,l \in V$. By Lemma \ref{g}, we know $u_{ij} u_{kl} = u_{kl} u_{ij}$ for $d(i,k) = d(j,l) \leq 2$. Because of Lemma \ref{dist}, it just remains to prove $u_{ij} u_{kl} = u_{kl} u_{ij}$ for $d(i,k) = d(j,l) = 3$, to get that the Heawood graph has no quantum symmetry. But this follows from Lemma \ref{ia} (a), because we have $c_2=1$, $c_3=3$ and $b_1+1=b_0$. 

\subsection{The Pappus graph $(\{3,2,2,1;1,1,2,3\})$}
The Pappus graph has diameter four, and thus we have $d(i,k), d(j,l) \leq 4$ for all $i,j,k,l \in V$. We know $u_{ij} u_{kl} = u_{kl} u_{ij}$ for $d(i,k) = d(j,l) \leq 2$ because of Lemma \ref{g}. It holds $c_3 =2, c_4 =3$, $b_1+1=b_0$ and we can use Lemma \ref{ia} (a) two times to get $u_{ij} u_{kl} = u_{kl} u_{ij}$ for $3 \leq d(i,k) = d(j,l) \leq 4$. Using Lemma \ref{dist}, we conclude that the Pappus graph has no quantum symmetry. 

\subsection{The Desargues graph $(\{3,2,2,1,1;1,1,2,2,3\})$}
The Desargues graph has diameter five. Therefore we have $d(i,k), d(j,l) \leq 5$ for all $i,j,k,l \in V$. We know $u_{ij} u_{kl} = u_{kl} u_{ij}$ for $d(i,k) = d(j,l) \leq 2$ by Lemma \ref{g}. Using Lemma $\ref{dist}$, it remains to show $u_{ij} u_{kl} = u_{kl} u_{ij}$ for $3 \leq d(i,k) = d(j,l) \leq 5$. This follows from applying Lemma \ref{ia} (a) three times, since $c_3=2$, $c_4=2$, $c_5=3$ and $b_1+1=b_0$. Thus, the Desargues graph has no quantum symmetry.\\

To deal with more graphs we need an additional lemma. 

\begin{lem}\label{cubic}
Let $\Gamma$ be a cubic distance regular graph. If we know that $u_{ij}u_{kl} = u_{kl} u_{ij}$ for $d(i,k)=d(j,l) \leq m-1$ and it either holds 
\begin{itemize}
\item[(i)] $b_{m-1}=1$ or 
\item[(ii)] $b_{m-1}=2$ and $b_m=c_m=1$, girth $g(\Gamma)\geq 2m$,
\end{itemize}
then we get $u_{ij}u_{kl} = u_{kl} u_{ij}$ for $d(i,k)=d(j,l) = m$.
\end{lem}

\begin{proof}
 Let $d(i,k) = d(j,l) =m$. Let $t$ be a neighbor of $j$ in distance $m-1$ to $l$. Since we know that  $u_{ac} u_{bd} = u_{bd} u_{ac}$ for $d(a,b)= d(c,d) = m-1$, we get\allowdisplaybreaks
\begin{align*}
 u_{ij}u_{kl} = u_{ij}u_{kl} \sum_{\substack{p; d(l,p) =m,\\(t,p) \in E}} u_{ip},
\end{align*}
by Lemma \ref{abk}.

For (i), we have $b_{m-1} =1$ and we deduce that $j$ is the only neighbor of $t$ with $d(l,j)=m$, since $d(t,l)=m-1$. Therefore
\begin{align*}
 u_{ij}u_{kl} = u_{ij}u_{kl} \sum_{\substack{p; d(l,p) =m,\\(t,p) \in E}} u_{ip} = u_{ij} u_{kl} u_{ij},
\end{align*}
and Lemma \ref{com} yields $u_{ij} u_{kl} = u_{kl} u_{ij}$ for $d(i,k)=d(j,l) =m$. 

Regarding (ii), we have $b_{m-1} =2$. Thus there are two neighbors of $t$ with $d(l,j)=m$, since $d(t,l)=m-1$. Those are $j$ and another vertex $q$. Therefore
\begin{align*}
 u_{ij}u_{kl} = u_{ij}u_{kl} \sum_{\substack{p; d(l,p) =m,\\(t,p) \in E}} u_{ip} = u_{ij} u_{kl} (u_{ij} + u_{jq}).
\end{align*}
It holds $b_m=c_m=1$ and since $\Gamma$ is a cubic graph, this implies that there is exactly one neighbor, say $s$, of $k$ at distance $m$ to $i$. Similarly, we have neighbors $a$, $b$ of $l$ at distance $m$ to $j$, $q$ respectively. We deduce 
\begin{align*}
u_{ij}u_{kl}u_{iq} = u_{ij}u_{sa}u_{kl}u_{sb}u_{iq} = u_{ij}u_{sa}u_{sb}u_{kl}u_{iq} 
\end{align*}
by Relations \eqref{QA2}, \eqref{QA3} and since $u_{kl}u_{sb} = u_{sb}u_{kl}$. Assume $a=b$. Then we know $d(a,j)=m=d(a,q)$. We also have $d(a,t) =m$, since we know $d(l,t)=m-1$, $(l,a) \in E$ and $\Gamma$ has girth $g(\Gamma) \geq 2m$. But then $t$ has two neighbors at distance $m$ to $a$, namely $j$ and $q$. This contradicts the fact that there is exactly one neighbor of $t$ at distance $m$ to $a$. This yields $a \neq b$ and therefore
\begin{align*}
u_{ij}u_{kl}u_{iq}=u_{ij}u_{sa}u_{sb}u_{kl}u_{iq} =0.
\end{align*}
Summarizing, we get 
\begin{align*}
 u_{ij}u_{kl} = u_{ij}u_{kl} \sum_{\substack{p; d(l,p) =3,\\(t,p) \in E}} u_{ip} = u_{ij} u_{kl} (u_{ij} + u_{jq}) =u_{ij} u_{kl} u_{ij}
\end{align*}
and now Lemma \ref{com} yields $u_{ij} u_{kl} = u_{kl} u_{ij}$ for $d(i,k)=d(j,l) =m$.
\end{proof}

\subsection{The Dodecahedron $(\{3,2,1,1,1;1,1,1,2,3\})$}
The Dodecahedron has diameter five, thus we have $d(i,k),d(j,l) \leq 5$ for all $i,j,k,l \in V$. Lemma \ref{g} yields $u_{ij}u_{kl}=u_{kl}u_{ij}$ for $d(i,k)=d(j,l)\leq 2$. We get $u_{ij} u_{kl} = u_{kl} u_{ij}$ for $d(i,k)=d(j,l) =3$ by Lemma \ref{cubic} (i), since the Dodecahedron is a cubic distance regular graph with $b_2=1$.
Now, since we have $c_4 =2$, $c_5=3$ and $b_1+1=b_0$, we can use Lemma \ref{ia} (a) two times to get $u_{ij} u_{kl} = u_{kl} u_{ij}$ for $4 \leq d(i,k)=d(j,l) \leq 5$. We conclude that the Dodecahedron has no quantum symmetry by Lemma \ref{dist}. 

\subsection{The Coxeter graph $(\{3,2,2,1;1,1,1,2\})$}
Since the Coxeter graph has diameter four, we have $d(i,k),d(j,l) \leq 4$ for all $i,j,k,l \in V$. By Lemma \ref{g}, we know $u_{ij}u_{kl}=u_{kl}u_{ij}$ for $d(i,k)=d(j,l)\leq 2$. We have $b_2 =2, b_3 =1, c_3 =1$ in the intersection array of the Coxeter graph, where the Coxeter graph has girth $7$. Thus, we can use Lemma \ref{cubic} (ii) to get $u_{ij} u_{kl} = u_{kl} u_{ij}$ for $d(i,k)=d(j,l) =3$. We obtain $u_{ij} u_{kl} = u_{kl} u_{ij}$ for $d(i,k)=d(j,l) =4$ by Lemma \ref{ia} (a), since we have $c_4=2$ and $b_1+1=b_0$. Then Lemma \ref{dist} yields that the Coxeter graph has no quantum symmetry.\\

We give the following technical lemma because it applies to the three remaining graphs. 

\begin{lem}\label{d3}
Let $\Gamma$ be a cubic distance regular graph of order $\geq 10$ with $b_2=2$, $g(\Gamma)\geq 7$ and let $d(i,k)=d(j,l)=3$. Then
\begin{align*}
u_{ij}u_{kl}=u_{ij}u_{kl}(u_{ij} + u_{iq}),
\end{align*}
where $q$ is the unique vertex adjacent to the neighbor $x$ of $j$, $d(x,l)=2$ with $d(q,j)=2$ and $d(l,q)=3$.
\end{lem}

\begin{proof}
Let $d(i,k)=d(j,l)=3$.  
Let $x$ be the unique vertex with $(j,x) \in E, d(x,l) =2$. It is unique because we assumed $g(\Gamma)\geq 7$. By Lemma \ref{g}, we get $u_{kl}u_{yx}=u_{yx}u_{kl}$ for all $y \in V$ with $d(k,y)=2$. We obtain
 \begin{align*}
 u_{ij}u_{kl} = u_{ij}u_{kl} \sum_{\substack{p;(x,p) \in E,\\d(p,l)=3}} u_{ip} 
 \end{align*}
 by Lemma \ref{abk}. We conclude 
 \begin{align*}
 u_{ij}u_{kl}= u_{ij} u_{kl} (u_{ij} + u_{iq}),
 \end{align*}
 because $x$ has three neighbors where two of them are at distance three to $l$, since $b_2=2$. 
\end{proof}

\subsection{The Tutte $8$-cage $(\{3,2,2,2;1,1,1,3\})$}
The Tutte $8$-cage has diameter four, thus we have $d(i,k),d(j,l) \leq 4$ for $i,j,k,l \in V$. Lemma \ref{g} yields $u_{ij}u_{kl}=u_{kl}u_{ij}$ for $d(i,k)=d(j,l) \leq 2$.

Let $d(i,k)=d(j,l)=4$. The Tutte $8$-cage is the incidence graph of the Cremona-Richmond configuration, see \cite{Coxeter}. Therefore we can label one of the maximal independent sets by unordered $2$-subsets of $\{1, \dots, 6\}$, where vertices at distance two to the vertex $\{a,b\}$ are exactly those corresponding to a $2$-subset that does not contain $a$ or $b$. The remaining vertices in the maximal independent set are exactly the vertices in distance four to $\{a,b\}$. Using this labelling we write $j=\{1,2\}, l=\{1,3\}$ and show that 
\begin{align*}
u_{i\{1,2\}}u_{k \{1,3\}} = u_{k \{1,3\}} u_{i \{1,2\}}.
\end{align*}
This suffices to get $u_{ij}u_{kl}=u_{kl}u_{ij}$ for $d(i,k)=d(j,l)=4$ by Lemma \ref{disttrans}, because the Tutte $8$-cage is distance-transitive.\\

\noindent\emph{Step 1: We have $u_{i\{1,2\}}u_{k\{1,3\}}=u_{i\{1,2\}}u_{k\{1,3\}}(u_{i\{1,2\}}+ u_{i\{2,3\}})$.}\\
There are three vertices $t_a$, $a\in \{1,2,3\}$,  such that $d(j,t_a)=d(t_a,l)=2$, because we have $c_4=3$ and $c_3=1$. Since we know $u_{st_a}u_{kl} = u_{kl}u_{st_a}$ by Lemma \ref{d2}, we get
\begin{align*}
u_{ij}u_{kl} =u_{ij}u_{kl}\left(\sum_{\substack{p;d(p,l)=4,\\d(p,t_a)=2}}u_{ip}\right)
 \end{align*}
by Lemma \ref{abk}. 
Using this, we obtain \allowdisplaybreaks
\begin{align}\label{sameneighbors}
u_{ij}u_{kl} &=u_{ij}u_{kl}\left(\sum_{\substack{p;d(p,l)=4,\\d(p,t_1)=2}}u_{ip}\right)\left(\sum_{\substack{p;d(p,l)=4,\\d(p,t_2)=2}}u_{ip}\right)\left(\sum_{\substack{p;d(p,l)=4,\\d(p,t_3)=2}}u_{ip}\right)\nonumber\\
&=u_{ij}u_{kl}\left(\sum_{\substack{p;d(p,l)=4,\\d(p,t_a)=2, \,a=1,2,3}}u_{ip}\right).
\end{align}
The vertices in distance two to $\{1,2\}$ and $\{1,3\}$ are $t_1=\{4,5\}$, $t_2=\{4,6\}$ and $t_3= \{5,6\}$. Looking at Equation \eqref{sameneighbors}, we only have to consider vertices that are in distance two to those three vertices. The only $2$-subset of $\{1,\dots,6\}$ besides $\{1,2\}$ and $\{1,3\}$ that does not contain $4,5$ or $6$ is $\{2,3\}$. Thus we get
\begin{align*}
u_{i\{1,2\}}u_{k\{1,3\}}= u_{i\{1,2\}}u_{k\{1,3\}}(u_{i\{1,2\}} + u_{i\{2,3\}}).
\end{align*}

\noindent\emph{Step 2: We have $u_{i\{1,2\}}u_{k\{1,3\}} u_{i\{2,3\}}=0.$}\\
Using Relations \eqref{QA2} and \eqref{QA4}, we obtain
\begin{align*}
u_{i\{1,2\}}u_{k\{1,3\}}u_{i\{2,3\}} = u_{i\{1,2\}}\left(\sum_{\substack{s;d(s,i)=2,\\d(k,s)=4}}u_{s\{3,4\}}\right)u_{k\{1,3\}} \left(\sum_{\substack{t;d(t,i)=2,\\d(k,t)=4}}u_{t\{1,5\}}\right)u_{i\{2,3\}}.
\end{align*}
The vertex $\{1,3\}$ is the only one in distance four to $\{1,2\},\{2,3\},\{3,4\}$ and $\{1,5\}$, because the only pair of numbers where at least one of them is contained in those subsets are $1$ and $3$. Let $q \neq \{1,3\}$. If $d(q,\{1,5\}) \neq 4$ or $d(q,\{3,4\})\neq 4$, then 
\begin{align*}
 u_{i\{1,2\}}\left(\sum_{\substack{s;d(s,i)=2,\\d(k,s)=4}}u_{s\{3,4\}}\right)u_{kq} \left(\sum_{\substack{t;d(t,i)=2,\\d(k,t)=4}}u_{t\{1,5\}}\right)u_{i\{2,3\}} =0,
\end{align*}
by Lemma \ref{dist}. If we have $d(q,\{1,5\}) =4$ and $d(q,\{3,4\})= 4$, but \linebreak$d(q,\{1,2\})\neq 4$, we get 
\begin{align*}
 u_{i\{1,2\}}&\left(\sum_{\substack{s;d(s,i)=2,\\d(k,s)=4}}u_{s\{3,4\}}\right)u_{kq} \left(\sum_{\substack{t;d(t,i)=2,\\d(k,t)=4}}u_{t\{1,5\}}\right)u_{i\{2,3\}}\\ 
 &= u_{i\{1,2\}}u_{kq} \left(\sum_{\substack{t;d(t,i)=2,\\d(k,t)=4}}u_{t\{1,5\}}\right)u_{i\{2,3\}}=0,
\end{align*}
by using Relations \eqref{QA2}, \eqref{QA4} and Lemma \ref{dist}. A similar argument shows
\begin{align*}
u_{i\{1,2\}}\left(\sum_{\substack{s;d(s,i)=2,\\d(k,s)=4}}u_{s\{3,4\}}\right)u_{kq} \left(\sum_{\substack{t;d(t,i)=2,\\d(k,t)=4}}u_{t\{1,5\}}\right)u_{i\{2,3\}} =0
\end{align*}
for $q$ with $d(q,\{1,5\}) =4$, $d(q,\{3,4\})= 4$ and $d(q,\{2,3\}) \neq 4$.
This yields 
\begin{align*}
u_{i\{1,2\}}u_{k\{1,3\}}u_{i\{2,3\}} &= u_{i\{1,2\}}\left(\sum_{\substack{s;d(s,i)=2,\\d(k,s)=4}}u_{s\{3,4\}}\right)u_{k\{1,3\}} \left(\sum_{\substack{t;d(t,i)=2,\\d(k,t)=4}}u_{t\{1,5\}}\right)u_{i\{2,3\}}\\
&= u_{i\{1,2\}}\left(\sum_{\substack{s;d(s,i)=2,\\d(k,s)=4}}u_{s\{3,4\}}\right)\left(\sum_a u_{ka}\right) \left(\sum_{\substack{t;d(t,i)=2,\\d(k,t)=4}}u_{t\{1,5\}}\right)u_{i\{2,3\}}\\
&= u_{i\{1,2\}}\left(\sum_{\substack{s;d(s,i)=2,\\d(k,s)=4}}u_{s\{3,4\}}\right)\left(\sum_{\substack{t;d(t,i)=2,\\d(k,t)=4}}u_{t\{1,5\}}\right)u_{i\{2,3\}}.
\end{align*}
Since $d(\{1,2\},\{3,4\})=d(\{2,3\},\{1,5\})=2$,  we know that $u_{i\{1,2\}}$ commutes with $u_{s\{3,4\}}$ and $u_{i\{2,3\}}$ commutes with $u_{t\{1,5\}}$ by Lemma \ref{d2}. We deduce
\begin{align*}\allowdisplaybreaks
u_{i\{1,2\}}u_{k\{1,3\}}u_{i\{2,3\}} &=u_{i\{1,2\}}\left(\sum_{\substack{s;d(s,i)=2,\\d(k,s)=4}}u_{s\{3,4\}}\right)\left(\sum_{\substack{t;d(t,i)=2,\\d(k,t)=4}}u_{t\{1,5\}}\right)u_{i\{2,3\}}\\
&=\left(\sum_{\substack{s;d(s,i)=2,\\d(k,s)=4}}u_{s\{3,4\}}\right)u_{i\{1,2\}}u_{i\{2,3\}}\left(\sum_{\substack{t;d(t,i)=2,\\d(k,t)=4}}u_{t\{1,5\}}\right)\\
&=0,
\end{align*}
since $u_{i\{1,2\}}u_{i\{2,3\}}=0$.\\

\noindent\emph{Step 3: It holds $u_{i\{1,2\}}u_{k\{1,3\}}= u_{k \{1,3\}} u_{i\{1,2\}}.$}\\
Using \emph{Step 1} and \emph{Step 2}, we obtain $u_{i\{1,2\}}u_{k \{1,3\}} = u_{i\{1,2\}}u_{k \{1,3\}} u_{i \{1,2\}}$. By Lemma \ref{com}, we get that  $u_{i\{1,2\}}$ and $u_{k \{1,3\}}$ commute. \\

\noindent\emph{Step 4: It holds $u_{ij}u_{kl}= u_{kl} u_{ij}$ for $d(i,k)=d(j,l)=3$.}\\
Let $d(i,k)=d(j,l)=3$. We prove that $u_{ij}u_{kl} = u_{kl}u_{ij}$ for $d(i,k)=d(j,l)=3$. Let $x$ be the unique vertex adjacent to $j$ and in distance two to $l$. This vertex is unique because the Tutte $8$-cage has girth eight. By Lemma \ref{d3}, we get 
\begin{align*}
u_{ij}u_{kl}=u_{ij}u_{kl}(u_{ij} + u_{iq}),
\end{align*}
where $q$ is the unique vertex adjacent to $x$ with $d(q,j)=2, d(l,q)=3$. Take a neighbor $t$ of $j$ at distance four to $l$. Then we have
\begin{align*}
u_{ij}u_{kl}u_{iq}= u_{ij} \left(\sum_{\substack{s;(s,i) \in E,\\d(k,s)=4}}u_{st}\right)u_{kl}u_{iq} = u_{ij} u_{kl}\left(\sum_{\substack{s;(s,i) \in E,\\d(k,s)=4}}u_{st}\right)u_{iq}, 
\end{align*}
because of Relations \eqref{QA2}, \eqref{QA3} and $u_{kl}u_{st} = u_{st}u_{kl}$ for all such $s$ since $d(t,l)=d(s,k)=4$. Assume that $t$ is connected to $q$. Then $j$ and $q$ have two common neighbors, $x$ and $t$, where we know that $x\neq t$ because we have $d(x,l)=2$ whereas $d(t,l)=4$. But then we get the quadrangle $j,x,q,t,j$ and this contradicts the fact that the Tutte $8$-cage has girth eight. Thus, $t$ and $q$ are not adjacent. We deduce 
\begin{align*}
u_{ij}u_{kl}u_{iq}= u_{ij} u_{kl}\left(\sum_{\substack{s;(s,i) \in E,\\d(k,s)=4}}u_{st}\right)u_{iq}=0
\end{align*}
by Relation \eqref{QA3}. Thus we get $u_{ij}u_{kl} = u_{ij}u_{kl}u_{ij}$ and we obtain $u_{ij}u_{kl}=u_{kl}u_{ij}$ for $d(i,k)=d(j,l)=3$ by Lemma \ref{com}. 

Summarizing, we have $u_{ij}u_{kl}=u_{kl}u_{ij}$ for $d(i,k)=d(j,l)\leq 4$. Using Lemma \ref{dist}, we conclude that the Tutte $8$-cage has no quantum symmetry. 

\subsection{The Foster graph $(\{3,2,2,2,2,1,1,1;1,1,1,1,2,2,2,3\})$}
The Foster graph has diameter eight. Therefore, we have $d(i,k), d(j,l) \leq 8$ for $i,j,k,l \in V$. By Lemma \ref{g}, we know $u_{ij}u_{kl}=u_{kl}u_{ij}$ for $d(i,k)=d(j,l)\leq 2$. 

 Let $d(i,k)=3$, $d(j,l)=3$. We want to show $u_{ij}u_{kl}=u_{kl}u_{ij}$.\\
 
\noindent\emph{Step 1: It holds $u_{ij}u_{kl}=u_{ij}u_{kl}(u_{ij} + u_{iq})$, where $q$ is the unique vertex adjacent to the neighbor $x$ of $j$, $d(x,l)=2$ with $d(q,j)=2, d(l,q)=3$.}\\
The Foster graph has girth ten. Thus, by Lemma \ref{d3}, we get
\begin{align*}
u_{ij}u_{kl}=u_{ij}u_{kl}(u_{ij} + u_{iq}),
\end{align*}
for $q$ as above.\\

\noindent\emph{Step 2: It holds $u_{ij}u_{kl}u_{iq} = u_{ij}\left(\sum\limits_{\substack{s;(s,k) \in E,\\ d(s,i)=4}}u_{sz}\right)u_{iq}$ for $z\in V$ with $(z,l) \in E$ and $d(z,j)=4$.}\\
Take $z$ with $(z,l) \in E, d(z,j)=4$. Using Relations \eqref{QA2} and \eqref{QA4}, we obtain
\begin{align*}
u_{ij}u_{kl}u_{iq} = u_{ij}\left(\sum_{\substack{s;(s,k) \in E,\\ d(s,i)=4}}u_{sz}\right)u_{kl}u_{iq}.
\end{align*}
We know $(z,l)\in E$ and $d(l,q)=3$. Since we have $b_3=2$ and $c_3=1$, it either holds $d(z,q)=2$ or $d(z,q)=4$. Assume $d(z,q)=2$. Then we get a cycle of length $\leq 6$, since we have $d(z,q)=2$, $d(q,l)=3$ and $(z,l) \in E$. But this contradicts the fact that the Foster graph has girth ten and we conclude $d(z,q)=4$. It holds $c_4=1$, thus $l$ is the only neighbor of $z$ at distance three to $q$. This yields
\begin{align*}\allowdisplaybreaks
u_{ij}\left(\sum_{\substack{s;(s,k) \in E,\\ d(s,i)=4}}u_{sz}\right)u_{kl}u_{iq} &= u_{ij}\left(\sum_{\substack{s;(s,k) \in E,\\ d(s,i)=4}}u_{sz}\right)\left(\sum_{\substack{a;(a,z) \in E,\\d(a,q)=3}}u_{ka}\right)u_{iq}\\
 &= u_{ij}\left(\sum_{\substack{s;(s,k) \in E,\\ d(s,i)=4}}u_{sz}\right)\left(\sum_{a}u_{ka}\right)u_{iq}\\
 &= u_{ij}\left(\sum_{\substack{s;(s,k) \in E,\\ d(s,i)=4}}u_{sz}\right)u_{iq},
\end{align*}
by using Lemma \ref{dist} and Relations \eqref{QA2}, \eqref{QA3}.\\

\noindent\emph{Step 3: It holds $u_{ij}u_{sz}u_{iq}=0$ for $s \in V$ with $(s,k) \in E$ and $d(i,s) =4$.}\\
For every such $s$, take $t$ with $d(i,t)=4, d(s,t)=2$ (exists because for a neighbor of $s$ at distance five to $i$ there is a neighbor $t \neq s$ with $d(i,t)=4$, since $c_5=2$). We get
\begin{align}\label{eqd3}
u_{ij}u_{sz}u_{iq}= u_{ij} \left(\sum_{\substack{p;d(p,z)=2,\\d(p,j)=4}}u_{tp}\right)u_{sz}u_{iq} = u_{ij} u_{sz}\left(\sum_{\substack{p;d(p,z)=2,\\d(p,j)=4}}u_{tp}\right)u_{iq}
\end{align}
by Lemma \ref{dist} and because we have $u_{tp}u_{sz}= u_{sz}u_{tp}$ by Lemma \ref{d2}. 
The Foster graph has girth ten, therefore there is exactly one neighbor of $j$ at distance two to $l$. This is the vertex $x$ from \emph{Step 1}. We know that $q$ is also a neighbor of $x$. Take $p$ with $d(p,z)=2$, $d(p,j)=4$. We want to show $d(p,q)\neq 4$. For this, we assume $d(p,q) = 4$ and prove that then $x$ has three neighbors in distance four to $p$, contradicting $c_5=2$. We have $d(p,j)=4$ by the choice of $p$. Let $y$ be the third neighbor of $x$.  We have $d(x,l)=2$ and know $d(j,l)=d(q,l)=3$, therefore the remaining neighbor $y$ of $x$ has to be adjacent to $l$ as otherwise $d(x,l) \neq 2$. We also have $(z,l) \in E$ and $d(z,p)=2$, where we know $z \neq y$ since $d(x,j)=2$, but $d(z,j)=4$. It holds $d(l,p)=3$ and we have $c_3=1$, $b_3=2$. Thus, $l$ has one neighbor at distance two to $p$ and two neighbors at distance four to $p$. Since we know that $z$ is the neighbor of $l$ with $d(z,p)=2$, we conclude $d(y,p)=4$ because $y$ is another neighbor of $l$. Thus $x$ has the three neighbors $j, q$ and $y$ in distance four to $p$ contradicting $c_5=2$. We conclude $d(p,q)\neq 4$ for all $p$ with $d(p,z)=2$, $d(p,j)=4$. By Lemma \ref{dist}, we deduce 
\begin{align*}
u_{ij}u_{sz}u_{iq}=u_{ij} u_{sz}\left(\sum_{\substack{p;d(p,z)=2,\\d(p,j)=4}}u_{tp}\right)u_{iq}=0
\end{align*}
for all $s$ with $(s,k) \in E, d(s,i)=4$. \\

\noindent\emph{Step 4: It holds $u_{ij}u_{kl}=u_{kl}u_{ij}$ for $d(i,k)=d(j,l)=3$.}\\
The Steps \emph{2} and \emph{3} yield
\begin{align*}
u_{ij}u_{kl}u_{iq}=u_{ij}\left(\sum_{\substack{s;(s,k) \in E,\\ d(s,i)=4}}u_{sz}\right)u_{iq}=0. 
\end{align*}
Using \emph{Step 1}, we get $u_{ij}u_{kl}=u_{ij}u_{kl}u_{ij}$ for $d(i,k)=d(j,l)=3$ and we obtain $u_{ij}u_{kl}=u_{kl}u_{ij}$ by Lemma \ref{com}.\\
 
 \noindent\emph{Step 5: We have $u_{ij}u_{kl}=u_{kl}u_{ij}$ for $4 \leq d(i,k)=d(j,l) \leq 8$.}\\
Let $d(i,k)=d(j,l)=4$. There is exactly one $a$ with $(a,i) \in E$, $d(k,a)=3$, and exactly one $b$ with $(b,j) \in E$, $d(l,b)=3$ since $c_4=1$. It holds
 \begin{align*}
 u_{ij}u_{kl} = u_{ij}u_{kl}\left(\sum_{\substack{p;d(p,l)=4,\\(b,p)\in E}}u_{ip}\right)
 \end{align*}
 by Lemma \ref{abk} since $u_{kl}u_{ab} =u_{ab}u_{kl}$. There are exactly two vertices adjacent to $b$ and at distance four to $l$, since $d(b,l)=3$ and $b_3=2$. One of them is $j$, so we get
 \begin{align*}
 u_{ij}u_{kl}= u_{ij}u_{kl}(u_{ij}+u_{iq}),
 \end{align*}
 where $q$ is the other neighbor of $b$ in distance four to $l$. We have 
 \begin{align*}
 u_{ij}u_{kl}u_{iq} = u_{ij}\left(\sum_{\substack{p;d(p,l)=2,\\d(p,j)=4}}u_{tp}\right)u_{kl}u_{iq} = u_{ij}u_{kl}\left(\sum_{\substack{p;d(p,l)=2,\\d(p,j)=4}}u_{tp}\right)u_{iq}
 \end{align*}
 by Relations \eqref{QA2}, \eqref{QA3} and $u_{kl}u_{tp}=u_{tp}u_{kl}$. But now we are in the same situation as in Equation \eqref{eqd3}, thus by the same argument we get $d(p,q) \neq 4$. By Lemma \ref{dist}, we deduce $u_{ij}u_{kl}u_{iq}=0$. This implies $u_{ij}u_{kl} = u_{ij}u_{kl}u_{ij}$ and Lemma \ref{com} yields $u_{ij}u_{kl}=u_{kl}u_{ij}$ for $d(i,k)=d(j,l)=4$. 
 
 Since we now know $u_{ij}u_{kl}=u_{kl}u_{ij}$ for $d(i,k),d(j,l) \leq 4$ and it holds $c_2=1$, $b_1+1=b_0$ and $c_n \geq 2$ for $5\leq n \leq 8$, we can use Lemma \ref{ia} (a) four times to get $u_{ij}u_{kl}=u_{kl}u_{ij}$ for $d(i,k)=d(j,l) \leq 8$. Then Lemma \ref{dist} yields that the Foster graph has no quantum symmetry.

\subsection{The Biggs-Smith graph $(\{3,2,2,2,1,1,1;1,1,1,1,1,1,3\})$}
Since the Biggs-Smith graph has diameter seven, we have $d(i,k), d(j,l) \leq 7$ for $i,j,k,l \in V$. By Lemma \ref{g}, we get $u_{ij}u_{kl}=u_{kl}u_{ij}$ for $d(i,k)=d(j,l)\leq 2$.

Let $d(i,k)=d(j,l)=4$. We show $u_{ij}u_{kl}=u_{kl}u_{ij}$ for $d(i,k)=d(j,l)=4$.\\

\noindent\emph{Step 1: It holds $u_{ij}u_{kl}= u_{ij}u_{kl}(u_{ij} + u_{iq_1} + u_{iq_2} + u_{iq_3})$, where $d(j,q_1)=2$ and $d(j,q_2)=d(j,q_3)=4$.}\\
 Since the Biggs-Smith graph has girth nine, there is exactly one vertex $t$ with $d(t,j)=d(t,l)=2$ and exactly one vertex $s$ with $d(i,s)=d(s,k)=2$. It holds
\begin{align*}
u_{ij}u_{kl} =u_{ij}u_{kl}\left(\sum_{\substack{p;d(p,l)=4,\\d(p,t)=2}}u_{ip}\right)
\end{align*}
by Lemma \ref{abk}, since we know $u_{kl}u_{at}=u_{at}u_{kl}$ for $a$ with $d(a,l)=2$. 
There are exactly four vertices that are at distance four to $l$ and at distance two to $t$, where one of them is $j$ (There are six vertices at distance two to $t$, where one of them is $l$ and another one is a vertex at distance two to $l$. The rest is at distance four to $l$). We deduce
\begin{align*}
u_{ij}u_{kl}= u_{ij}u_{kl}(u_{ij} + u_{iq_1} + u_{iq_2} + u_{iq_3}),
\end{align*}
where $d(j,q_1)=2$ and $d(j,q_2)=d(j,q_3)=4$.\\

\noindent\emph{Step 2: We have $u_{ij}u_{kl}u_{iq_1}=0$.}\\
We know that the Biggs-Smith is $3$-regular and it holds $b_4=c_4=1$. Thus, if we have $d(a,b)=4$ for vertices $a,b$, there is exactly one neighbor of $a$ in distance four to $b$. Therefore we have exactly one neighbor $x$ of $k$ with $d(x,i)=4$ and exactly one neighbor $y$ of $l$ with $d(y,j)=4$. By Relations \eqref{QA2}, \eqref{QA3}, we deduce
\begin{align*}
u_{ij}u_{kl}u_{iq_1} = u_{ij}u_{xy}u_{kl}u_{iq_1}.
\end{align*}
Denote by $z_1$ the common neighbor of $j$ and $q_1$. We know that $j$ and $q_1$ are two neighbors of $z_1$ at distance four to $l$. Thus $d(z_1,l)=4$ contradicts $c_4=b_4=1$ and $d(z_1,l)=5$ contradicts $c_5=1$. But it holds $d(z_1,l) \in \{3,4,5\}$ since $z_1$ has neighbors in distance four to $l$. We deduce $d(z_1,l)=3$. We have $b_3=2$, $c_3=1$ and know $d(z_1,l)=3$, thus $l$ has two neighbors at distance four to $z_1$ and one neighbor at distance two to $z_1$. If $d(y,z_1)=2$, then we get $d(y,j) \leq 3$ since $(j,z_1) \in E$ contradicting $d(y,j)=4$. We conclude $d(y,z_1)=4$ as $y$ is a neighbor of $l$ not in distance two to $z_1$. Furthermore it holds $d(q_1,y)\neq4$, as otherwise $z_1$ would have the two neighbors $j$ and $q_1$ at distance four to $y$ contradicting $b_4=c_4=1$, since $d(z_1,y)=4$. But this yields
\begin{align*}
u_{ij}u_{kl}u_{iq_1} = u_{ij}u_{xy}u_{kl}u_{iq_1}=u_{ij}u_{kl}u_{xy}u_{iq_1}=0, 
\end{align*}
by using Relation \eqref{QA3} and $u_{xy}u_{kl}=u_{kl}u_{xy}$. \\

\noindent\emph{Step 3: We have $u_{ij}u_{kl}u_{iq_2}=0$.}\\
We know that $q_2$ and $q_3$ are in distance two to $t$. Thus, they have to be adjacent to one of the neighbors of $t$. They cannot be adjacent to $z_1$, because $z_1$ has neighbors $j,q_1$ and $t$ and the Biggs-Smith graph is $3$-regular. It holds $b_2=2$, $c_2=1$ and we know $d(t,l)=2$, which means that $t$ has one neighbor, say $z_2$, adjacent to $l$ and two neighbors ($z_1$ and one more) in distance three to $l$. Denote the third neighbor of $t$ by $z_3$. Recall $d(q_2,l)=d(q_3,l)=4$. Thus $q_2$, $q_3$ cannot be adjacent to $z_2$ as otherwise $d(q_2,l)=d(q_3,l) \leq 2$. We conclude that $q_2$, $q_3$ are both neighbors of $z_3$, where $d(z_3,l)=3$. The vertices $z_2$ and $y$ are neighbors of $l$. We have $d(l,t)=2$ and $(t,z_2)\in E$, where we know that the other neighbors of $l$ are in distance three to $t$ because we have $b_2=2$. We deduce $d(y,t)=3$. It also holds $d(z_2,y)=2$ since they have the common neighbor $l$ and the Biggs-Smith graph has girth nine. Thus we get $d(z_3,y)=4$ by $b_3=2$, since $d(t,y)=3$ and we know that $z_1$ is the neighbor of $t$ in distance two to $y$. Because of $d(z_3, y)=4$ and $c_4=b_4=1$, we see that only one of the vertices $q_2$, $q_3$ is in distance four to $y$, say this is $q_3$. We obtain 
\begin{align*}
u_{ij}u_{kl}u_{iq_2} = u_{ij}u_{xy}u_{kl}u_{iq_2} = u_{ij}u_{kl}u_{xy}u_{iq_2} =0,
\end{align*}
by Relations \eqref{QA2}, \eqref{QA3}, using $u_{xy}u_{kl}=u_{kl}u_{xy}$ and Lemma \ref{dist}. \\

\noindent\emph{Step 4: It holds $u_{ij}u_{kl}u_{iq_3}=0$.}\\
We have $d(q_3,y)=4$ and since $l$ is a neighbor of $y$ at distance four to $q_3$, we know that the two neighbors $c,d \neq l$ of $y$ are not in distance four to $q_3$ because $c_4=b_4=1$. Therefore
\begin{align*}
u_{ij}u_{xy}u_{kc}u_{iq_3} = 0 =u_{ij}u_{xy}u_{kd}u_{iq_3}
\end{align*}
by Lemma \ref{dist}. We deduce 
\begin{align*}
u_{ij}u_{kl}u_{iq_3} &= u_{ij}u_{xy}u_{kl}u_{iq_3} \\
&= u_{ij}u_{xy}(u_{kl} + u_{kc}+ u_{kd})u_{iq_3} \\
&= u_{ij}u_{xy}\left(\sum_{a;(y,a) \in E}u_{ka}\right)u_{iq_3} \\
&=u_{ij}u_{xy}\left(\sum_{a}u_{ka}\right)u_{iq_3} \\
&= u_{ij}u_{xy}u_{iq_3},
\end{align*}
by also using Relations \eqref{QA2}, \eqref{QA3}. Now, take $e,f$ to be the vertices with $d(e,i)=d(e,x)=2$, $d(f,j)=d(f,y)=2$ (those are unique, since $d(i,x)=d(j,y)=4$ and the Biggs-Smith graph has girth nine). It holds
\begin{align*}
u_{ij}u_{xy}u_{iq_3}= u_{ij}u_{ef}u_{xy}u_{iq_3} = u_{ij}u_{xy}u_{ef}u_{iq_3},
\end{align*}
by Relations \eqref{QA2}, \eqref{QA3} and since we know $u_{xy}u_{ef}=u_{ef}u_{xy}$ by Lemma \ref{g}. We have $d(f,q_3) \neq 2$ because otherwise there would be two vertices, $f$ and $t$, in distance two to $j$ and $q_3$, so we would get an cycle of length $\leq 8$ in the Biggs-Smith graph($f \neq t$, since $d(y,f)=2, d(y,t)=3$). Thus, by Lemma \ref{dist}, we get
\begin{align*}
u_{ij}u_{kl}u_{iq_3}=  u_{ij}u_{xy}u_{iq_3}= u_{ij}u_{xy}u_{ef}u_{iq_3} =0.
 \end{align*}
 
 \noindent\emph{Step 5: It holds $u_{ij}u_{kl}=u_{kl}u_{ij}$ for $d(i,k)=d(j,l)=4$.}\\
 From Steps \emph{1}--\emph{4}, we deduce $u_{ij}u_{kl}=u_{ij}u_{kl}u_{ij}$ for $d(i,k)=d(j,l)=4$ and we get $u_{ij}u_{kl}=u_{kl}u_{ij}$ by Lemma \ref{com}.\\
 
\noindent\emph{Step 6: We have $u_{ij}u_{kl}=u_{kl}u_{ij}$ for $d(i,k)=d(j,l)=3$.}\\
Let $d(i,k)=d(j,l)=3$. We obtain
 \begin{align*}
 u_{ij}u_{kl} = u_{ij} u_{kl} (u_{ij} + u_{iq}),
 \end{align*}
by Lemma \ref{d3}, where $d(q,j)=2, d(l,q)=3$. We have $b_3=2$ and therefore there are two neighbors $t_1,t_2$ of $j$ in distance four to $l$ and two neighbors $s_1, s_2$ of $i$ in distance four to $k$. At least one of them, say $t_1$, is not connected to $q$, since otherwise we would get the quadrangle $j,t_1,q,t_2,j$. By Lemma \ref{dist} we get $u_{s_at_1}u_{iq} =0$, $a=1,2$. Because we know $u_{s_at_1}u_{kl} = u_{kl} u_{s_at_1}$, since $d(s_a,k) = 4 = d(t_1,l)$, we deduce
\begin{align*}
 u_{ij}u_{kl}u_{iq} = u_{ij}(u_{s_1t_1} + u_{s_2t_1})u_{kl}u_{iq} = u_{ij}u_{kl}(u_{s_1t_1} + u_{s_2t_1})u_{iq} =0.
\end{align*}
This yields $u_{ij}u_{kl}=u_{ij}u_{kl}u_{ij}$ for $d(i,k)=d(j,l)=3$ and we obtain $u_{ij}u_{kl}=u_{kl}u_{ij}$ by Lemma \ref{com}.\\

\noindent\emph{Step 7: We have $u_{ij}u_{kl}=u_{kl}u_{ij}$ for $5 \leq d(i,k)=d(j,l) \leq 7$.}\\
We now have $u_{ij}u_{kl}=u_{kl}u_{ij}$ for $d(i,k)=d(j,l)\leq 4$ and since $b_4=1$, we get $u_{ij}u_{kl}=u_{kl}u_{ij}$ for $d(i,k)=d(j,l) =5$ by Lemma \ref{cubic} (i). We have $b_5=1$ and thus, using Lemma \ref{cubic} (i) again, we obtain  $u_{ij}u_{kl}=u_{kl}u_{ij}$ for $d(i,k)=d(j,l) =6$. Lemma \ref{ia} (a) now yields $u_{ij}u_{kl}=u_{kl}u_{ij}$ for $d(i,k)=d(j,l) =7$, because $c_2=1$, $b_1+1=b_0$, $c_7=3$. Using Lemma \ref{dist}, we conclude that the Biggs-Smith graph has no quantum symmetry.

\begin{rem}
There is only one cubic distance-regular graph that is not distance-transitive. This is the Tutte $12$-cage. We do not know whether or not this graph has quantum symmetry. 
\end{rem}

\section{Further distance-regular graphs with no quantum symmetry}\label{further}

In this chapter, we study further distance-regular graphs of order $\leq 20$. We assume $(u_{ij})_{1 \leq i,j \leq n}$ to be the generators of $C(\QBan(\Gamma))$ and show that the graph $\Gamma$ of the has no quantum symmetry in the corresponding subsection.

The following graph is the co-Heawood graph, which is the bipartite complement of the Heawood graph with respect to the complete bipartite graph $K_{7,7}$ and thus closely related to the Heawood graph. 

\subsection{The co-Heawood graph $(\{4,3,2;1,2,4\})$}
The co-Heawood graph has diameter three. Therefore we have $d(i,k),d(j,l) \leq 3$ for $i,j,k,l \in V$. Since the co-Heawood graph is the bipartite complement of the Heawood graph with respect to $K_{7,7}$, we see that vertices at distance three to a vertex $i$ are exactly those that are connected to $i$ in the Heawood graph. Vertices at distance two are the same ones in both graphs, since those are the six other vertices in the same maximal independent set as $i$. And finally the vertices that are connected to $i$ in the co-Heawood graph are those at distance three to $i$ in the Heawood graph. 
Therefore, we can use the same arguments as in Theorem \ref{g5} to obtain that $u_{ij} u_{kl} = u_{kl} u_{ij}$ for $d(i,k) = d(j,l)=3$ for the co-Heawood graph. Also arguments of the proof of Lemma \ref{d2} work similarly to show $u_{ij} u_{kl} = u_{kl} u_{ij}$ for $d(i,k) = d(j,l)=2$ by replacing neighbors with vertices at distance three. Then using the same approach as in Lemma \ref{ia} (a), also replacing neighbors with vertices at distance three, we get $u_{ij} u_{kl} = u_{kl} u_{ij}$ for $(i,k), (j,l)\in E$. We obtain that the co-Heawood graph has no quantum symmetry by Lemma \ref{dist}.

\subsection{The line graph of the Petersen graph $L(\mathrm{P})$ $(\{4,2,1;1,1,4\})$}
The line graph of the Petersen graph has diameter three and thus we have $d(i,k),d(j,l) \leq 3$ for $i,j,k,l \in V$. Since adjacent vertices have exactly one common neighbor, Lemma \ref{oneneighbor} yields $\QBan(L(\mathrm{P})) = \QBic(L(\mathrm{P}))$. Therefore Relation \eqref{QA5} holds. 

Now, let $d(i,k)=d(j,l) =2$. We want to prove $u_{ij}u_{kl}=u_{kl}u_{ij}$. We know that the Petersen graph is the Kneser graph $K(5,2)$. Thus, vertices in the line graph of the Petersen graph are of the form $\{\{a,b\},\{c,d\}\}$, where $\{a,b\}$, $\{c,d\}$ are disjoint $2$-subsets of $\{1,\dots, 5\}$. Two vertices are connected if and only if they have exactly one $2$-subset in common. The line graph of the Petersen graph is distance-transitive, therefore it suffices to show $u_{ij}u_{kl}= u_{kl} u_{ij}$ for $j = \{\{1,2\},\{3,4\}\}, l = \{\{1,3\}, \{4,5\}\}$ by Lemma \ref{disttrans}. The only common neighbor of $j$ and $l$ is $t= \{\{1,2\},\{4,5\}\}$. Since we know $\QBan(L(\mathrm{P})) = \QBic(L(\mathrm{P}))$, we get
\begin{align*}
u_{ij}u_{kl} = u_{ij}u_{kl} \sum_{\substack{p;d(p,l)=2,\\ (p,t) \in E}} u_{ip}
\end{align*}
by Lemma \ref{abk}.
 Besides $j$, the vertex $q = \{\{1,2\},\{3,5\}\}$ is the only other vertex in distance two to $l$ which is also adjacent to $t$. This yields
 \begin{align*}
 u_{ij}u_{kl} = u_{ij}u_{kl}(u_{ij} + u_{iq}).
 \end{align*}
 The vertex $b=\{\{1,3\}, \{2,4\}\}$ is adjacent to $l$ and in distance three to $j$. Using Relations \eqref{QA2} and \eqref{QA4}, we deduce
 \begin{align*}
 u_{ij}u_{kl}u_{iq} = u_{ij}\left(\sum_{\substack{a;d(a,i)=3,\\ (a,k) \in E}} u_{ab}\right)u_{kl}u_{iq}. 
 \end{align*}
 Because of Relation \eqref{QA5}, we get 
 \begin{align*}
 u_{ij}u_{kl}u_{iq} = u_{ij}u_{kl}\left(\sum_{\substack{a;d(a,i)=3,\\ (a,k) \in E}} u_{ab}\right)u_{iq}. 
 \end{align*}
 We see that $\{\{2,4\},\{3,5\}\}$ is a neighbor of $b$ and $q$. This yields $d(b,q) \leq2$ and since we have $d(a,i)=3$, we obtain
 \begin{align*}
 u_{ij}u_{kl}u_{iq} = u_{ij}u_{kl}\left(\sum_{\substack{a;d(a,i)=3,\\ (a,k) \in E}} u_{ab}\right)u_{iq}= 0,
 \end{align*}
 by Lemma \ref{dist}. Summarizing, it holds $u_{ij}u_{kl}=u_{ij}u_{kl}u_{ij}$. By Lemma \ref{com}, we see that $u_{ij}$ and $u_{kl}$ commute. 
 
For $d(i,k)=d(j,l) =3$, all conditions for Lemma \ref{ia} (b) are fulfilled and we get $u_{ij}u_{kl}=u_{kl}u_{ij}$ for $d(i,k)=d(j,l) =3$. Using Lemma \ref{dist}, we deduce that the line graph of the Petersen graph has no quantum symmetry.

\begin{lem}\label{twocommon}
Let $\Gamma$ be an undirected graph with clique number three, where adjacent vertices and vertices at distance two have exactly two common neighbors. Then we have $\QBan(\Gamma) = \QBic(\Gamma)$.  
\end{lem}

\begin{proof}
Let $(i,k), (j,l) \in E$. By Relations \eqref{QA2}, \eqref{QA3} we have 
\begin{align*}
u_{ij}u_{kl} = u_{ij} u_{kl} \left(\sum_{p;(l,p) \in E} u_{ip}\right).
\end{align*} 
Denote the two common neighbors of $j$ and $l$ by $p_1, p_2$. 

We have $(p_1,p_2) \notin E$, since otherwise we get a clique of size four, but we know that the clique number of $\Gamma$ is three. Also $p_1,p_2$ have two common neighbors since $d(p_1,p_2)=2$, where we know that those are $l$ and $j$. This yields that $l$ is the only common neighbor of $p_1,p_2$ and $j$. We also have $(j,p_1) \in E$, $(p_1,p_2) \notin E$ by previous considerations and deduce
\begin{align*}
u_{ij} u_{kl} u_{ip_a} =0, \quad a=1,2
\end{align*}
by Lemma \ref{krit}, where we choose $q=p_1$ for $p_2$ and vice versa.

Now, let $p\notin \{j,p_1,p_2\}$ and $(l,p) \in E$ (this implies $(p,j) \notin E$). We know that we have $(p_1,p) \notin E$ or $(p_2,p) \notin E$ since otherwise $p_1$ and $p_2$ have three common neighbors: $j$, $l$ and $p$. Choose $p_x$, $x\in \{1,2\}$ such that $(p_x,p) \notin E$. Since $p_x,p$ have $l$ as common neighbor and we know $d(p_x,p)=2$, there is exactly one other common neighbor $q \neq l$ of $p_x,p$. It holds $(j,q) \notin E$, because otherwise $j,p_x$ and $p$ would be common neighbors of $l$ and $q$, but we know that they can only have two common neighbors since $d(l,q)\leq 2$. Therefore $l$ is the only common neighbor of $j,p_x$ and $p$. We also have $(j,p_x) \in E$, $(p_x,p) \notin E$ and we obtain
\begin{align*}
u_{ij}u_{kl}u_{ip}=0
\end{align*}
by Lemma \ref{krit}, where we choose $q=p_x$. 

Summarizing, we get 
\begin{align*}
u_{ij}u_{kl} = u_{ij} u_{kl} \left(\sum_{p;(l,p) \in E} u_{ip}\right)= u_{ij}u_{kl}u_{ij}
\end{align*} 
and by Lemma \ref{com} we get $u_{ij} u_{kl} = u_{kl} u_{ij}$ for $(i,k), (j,l) \in E$. 
\end{proof}

\subsection{The Icosahedron $(\{5,2,1;1,2,5\})$}
The Icosahedron has diameter three and therefore we have $d(i,k),d(j,l) \leq 3$ for $i,j,k,l \in V$. 
Since adjacent vertices and vertices at distance two have exactly two common neighbors and since we know that the clique number is three, we get 
$u_{ij} u_{kl} = u_{kl} u_{ij}$ for $(i,k), (j,l) \in E$ by Lemma \ref{twocommon} and $u_{ij} u_{kl} = u_{kl} u_{ij}$ for $d(i,k)=2 =d(j,l)$ by Lemma \ref{ia} (c). 

By Lemma \ref{dist} we know that $u_{ij}$ and $u_{kl}$ commute if $d(i,k)\neq d(j,l)$. Thus it remains to show $u_{ij} u_{kl} = u_{kl} u_{ij}$ for $d(i,k)=3 =d(j,l)$. Note that for every vertex $x$, there is exactly one other vertex at distance three to $x$. Let $d(i,k)=3 =d(j,l)$. By Lemma \ref{dist}, we get
\begin{align*}
u_{ij}u_{kl} = u_{ij} u_{kl} \sum_{p; d(l,p)=3} u_{ip}. 
\end{align*}
Since $j$ is the only vertex in distance three to $l$, we conclude
\begin{align*}
u_{ij}u_{kl}=u_{ij}u_{kl}u_{ij}.
\end{align*}
Then Lemma \ref{com} yields $u_{ij}u_{kl}=u_{kl}u_{ij}$ and we get that the Icosahedron has no quantum symmetry. 

\subsection{The Shrikhande graph $(\{6,3;1,2\})$}\label{nosym}
First note that the Shrikhande graph is strongly regular with parameters $(16,6,2,2)$. Thus it has diameter two and we know $d(i,k),d(j,l) \leq 2$ for $i,j,k,l \in V$. Since $\lambda=\mu=2$, we know that every two vertices have exactly two common neighbors and we also know that the clique number is three. By Lemma \ref{twocommon}, we obtain $u_{ij} u_{kl} = u_{kl} u_{ij}$ for $(i,k), (j,l) \in E$. Then all the conditions of Lemma \ref{ia} (c) are met, we get $u_{ij} u_{kl} = u_{kl} u_{ij}$ for $(i,k), (j,l) \notin E$. We conclude that the Shrikhande graph has no quantum symmetry.

\begin{rem}
The $4\times4$ rook's graph $H(2,4)$ is strongly regular with the same parameters as the Shrikhande graph, but we know that the $4\times4$ rook's graph has quantum symmetry by Proposition \ref{Hamm}. The proof above does not apply for the $4\times4$ rook's graph because this graph has cliques of size four. 
\end{rem}

\begin{rem}
Since the Shrikhande graph has no quantum symmetry, we get that the quantum orbital algebra and the classical orbital algebra are the same. Therefore the Shrikhande graph is a nice example of a graph whose quantum orbital algebra is different from the coherent algebra of the graph. See \cite{nonlocal} for more on quantum orbital algebras of graphs. 
\end{rem}

\bibliographystyle{plain}
\bibliography{disttransitive}
\end{document}